\documentclass[final]{siamltex}

\usepackage[notref,notcite]{showkeys}
\usepackage{verbatim}
\usepackage{amsfonts}
\usepackage{amsmath,amssymb}
\usepackage{graphicx}
\usepackage{enumerate,url}

\usepackage[vlined,ruled,nofillcomment,linesnumbered]{algorithm2e}
\SetAlgoSkip{}\SetAlgoInsideSkip{smallskip}

\newenvironment{algo}[1]
{\begin{algorithm}[#1]%
    \small%
    \DontPrintSemicolon%
    \SetArgSty{texttsf}%
    \SetTitleSty{textsf}{}%
    \SetNlSty{textrm}{}{}%
    \SetKwInput{Inputs}{input}%
    \SetKwInput{Outputs}{output}%
    \SetKwData{Converged}{converged}%
    \SetKwComment{tcc}{[}{]}
}
{\end{algorithm}}

\newcommand{\pmat}[1]{{\renewcommand{\arraystretch}{1.1}%
   \begin{pmatrix}#1\end{pmatrix}}}  % Need \usepackage{amsmath}
\newcommand{\bmat}[1]{{\renewcommand{\arraystretch}{1.1}%
   \begin{bmatrix}#1\end{bmatrix}}}

\ifx   \innerprod\undefined%
   \def\innerprod(#1,#2){\langle#1,#2\rangle} % Must be called as \innerproduct(A,B)
\fi

\newcommand{\Newcommand}[2]%
   {\ifx#1\undefined \newcommand{#1}{#2} \else \renewcommand{#1}{#2} \fi}
\ifx\mod\undefined
  \newcommand{\mod}[1]{|#1|}
\else
  \renewcommand{\mod}[1]{|#1|}
\fi
\Newcommand{\Re} {\mathbb{R}}           % Need \usepackage{amssymb}

\providecommand{\diag} {\mathop{\mathrm{diag}}}
\providecommand{\dim}  {\mathop{\mathrm{dim}}}

\providecommand{\rank} {\mathop{\mathrm{rank}}}

\providecommand{\CG}        {{\small CG}}
\providecommand{\CGLS}      {{\small CGLS}}

\providecommand{\MINRESQLP} {{\small MINRES-QLP}}

\providecommand{\SQMR}      {{\small SQMR}}
\providecommand{\SOL}       {{\small SOL}}
\providecommand{\LSQR}      {{\small LSQR}}
\newcommand{\SymOrtho}{\text{SymOrtho}}

\providecommand{\diag} {\mathop{\operator@font{diag}}}
\newcommand{\be}{\begin{enumerate}}
\newcommand{\ee}{\end{enumerate}}

\newcommand{\T}{^T\!}

\newcommand{\Cpp}{C\raise3pt\hbox{\tiny++}}

\newcommand{\cond}{\mathrm{cond}}

\newcommand{\Null}{\mathop{\mathrm{null}}}

\newcommand{\range}{\mathop{\mathrm{range}}}

\newcommand{\Span}{\mbox{\rm span}}

\newcommand{\etal}{et al.}  %%% No italics!!  Also, must say \etal\

\newcommand{\inv}{^{-1}}

\newcommand{\normm}[1]{\biggl\|#1\biggr\|}
\newcommand{\norm}[1]{\|#1\|}
\newcommand{\spose}[1]{\hbox to 0pt{#1\hss}}

\providecommand{\text}[1]{\hbox{\quad#1\quad}}

\newcommand{\nthinsp}{\mskip -2   mu}

\Newcommand{\R}{_{\scriptscriptstyle R}}

\newcommand{\superstar}{^{\raise 0.5pt\hbox{$\nthinsp *$}}}
\newcommand{\SUPERSTAR}{^{\raise 0.5pt\hbox{$*$}}}
\newcommand{\lamstarT }{\lambda^{\raise 0.5pt\hbox{$\nthinsp *$}T}}

\newcommand{\rhat}{\skew3\widehat r}

\newcommand{\xhat}{\skew{2.8}\widehat x}

\providecommand{\Matlab}{{\sc Matlab}}

\providecommand{\GMRES }{{\small GMRES}}
\providecommand{\LSQR  }{{\small LSQR}}
\providecommand{\MINRES}{{\small MINRES}}
\providecommand{\SYMMLQ}{{\small SYMMLQ}}

\newtheorem{example}{Example}[section]
\newtheorem{remark}{Remark}[section]

\newcommand{\e}[1]{\text{e}{#1}}
\newcommand{\myhalf}{\frac12}
\newcommand{\mystrut}{\rule[-1.9ex]{0pt}{5.3ex}}
\newcommand{\tablestrut}{\rule[-1ex]{0pt}{3.5ex}}
\newcommand{\underTj}{\underline{T_j}}
\newcommand{\underTk}{\underline{T_k}}
\newcommand{\underTkp}{\underline{T_{k+1}}}

\usepackage[pdftex,
  pdftitle={MINRES-QLP: a Krylov subspace method for
            indefinite or singular symmetric systems},
  pdfauthor={Sou-Cheng (Terrya) Choi, Christopher C. Paige, and Michael A. Saunders},
  pdfsubject={MINRES-QLP for indefinite or singular symmetric systems},
  pdfkeywords={MINRES, Krylov subspace method, Lanczos process,
    conjugate-gradient method, minimum-residual method, ill-posed
    problem, singular least-squares problem, sparse matrix},
  pdfpagemode=UseOutlines,
  pdfstartview=FitH,
  bookmarks,
  bookmarksopen,
  colorlinks,
  linkcolor=blue,
  citecolor=blue,
  urlcolor=red,
  hypertexnames=false
]{hyperref}

\title{MINRES-QLP: a Krylov subspace method for
       indefinite or singular symmetric systems%
    \thanks{Received by the editors March 7, 2010.
            \hspace*{200pt}\llap{Draft MINRESQLP56 of \today.}
            Revised March 31, 2011.
            \URL sisc/
            }}

\author{Sou-Cheng T. Choi%
    \thanks{Institute for Computational and Mathematical Engineering,
            Stanford University,
            Stanford, CA 94305-4121 (scchoi@stanford.edu).
            This author's research was partially supported by
            National Science Foundation grant CCR-0306662.}
\and
    Christopher C. Paige%
    \thanks{Computer Science, McGill University,
            Montreal, Quebec, Canada, H3A 2A7 (paige@cs.mcgill.ca).
            This author's research was partially supported by
            NSERC of Canada grant OGP0009236.}
\and
    Michael A. Saunders%
    \thanks{Department of Management Science and Engineering,
            Stanford University, Stanford, CA 94305-4026
            (saunders@stanford.edu).
            This author's research was partially supported by
            National Science Foundation grant CCR-0306662,
            Office of Naval Research grants N00014-02-1-0076
            and N00014-08-1-0191, and
            AHPCRC.}}

\begin{document}

\maketitle

\begin{abstract}
  CG, SYMMLQ, and MINRES are Krylov subspace methods for solving
  symmetric systems of linear equations.
  %CG (the conjugate-gradient
  %method) is reliable on positive-definite systems, while SYMMLQ and
  %MINRES are designed for indefinite systems.
  When these methods are applied to an incompatible system (that is, a
  singular symmetric least-squares problem), CG could break down and
  SYMMLQ's solution could explode, while MINRES would give a
  least-squares solution but not necessarily the minimum-length
  (pseudoinverse) solution. This
  understanding motivates us to design a MINRES-like algorithm to
  compute minimum-length solutions to singular symmetric systems.

  MINRES uses QR factors of the tridiagonal matrix from the Lanczos
  process (where $R$ is upper-tridiagonal).  MINRES-QLP
  uses a QLP decomposition (where rotations on the right reduce $R$ to
  lower-tridiagonal form). On ill-conditioned systems (singular or not),
  MINRES-QLP can give more accurate solutions than MINRES.  We derive
  preconditioned MINRES-QLP, new stopping rules, and better estimates
  of the solution and residual norms, the matrix norm, and the
  condition number.
\end{abstract}

\begin{keywords}
  MINRES, Krylov subspace method, Lanczos process,
  conjugate-gradient method, minimum-residual method, %SC: removed: ill-posed problem,
  singular least-squares problem, sparse matrix
\end{keywords}

\begin{AMS}
   15A06, 65F10, 65F20, 65F22, 65F25, 65F35, 65F50, 93E24
\end{AMS}

\begin{DOI}
   xxx/xxxxxxxxx
\end{DOI}

\section{Introduction}  \label{sec:intro}

We are concerned with iterative methods for solving a symmetric linear
system $Ax=b$ or the related least-squares (LS) problem
\begin{equation}  \label{eqn4b}
  \min \norm{x}_2 \quad \text{s.t.} \quad x \in \arg\min_x\norm{Ax-b}_2,
\end{equation}
where $A \in \mathbb{R}^{n \times n}$ is symmetric and possibly singular,
$b \in \mathbb{R}^{n}$, $A \neq 0$, and $b\neq 0$.
Most of the results in our discussion are directly extendable to problems
with complex Hermitian matrices $A$ and complex vectors $b$.

The solution of \eqref{eqn4b}, called the \textit{minimum-length} or
\textit{pseudoinverse} solution \cite{GV}, is formally given by
$x^\dagger = (A\T A)^\dagger A\T b=(A^2)^\dagger Ab = (A^\dagger)^2
Ab$, where $A^\dagger$ denotes the pseudoinverse of $A$. The
pseudoinverse is continuous under perturbations $E$ for which
$\rank{(A+E)}=\rank{(A)}$ \cite{S69}, and $x^\dagger$ is continuous
under the same condition.  Problem \eqref{eqn4b} is then well-posed
\cite{Had1902}.

Let $A=U\Lambda U\T $ be an eigenvalue decomposition of $A$, with $U$
orthogonal and $\Lambda \equiv \diag(\lambda_1,\ldots,\lambda_n)$.  We
define the condition number of $A$ to be $\kappa(A) =
\smash[b]{\frac{\max |\lambda_i|}{\min_{\lambda_i \neq 0} |\lambda_i|}}$,
and we say that $A$ is ill-conditioned if $\kappa(A)\gg 1$.
Hence a singular matrix could be well-conditioned or ill-conditioned.

\SYMMLQ{} and \MINRES{} \cite{PS75} are Krylov subspace methods for
solving symmetric indefinite systems $Ax = b$.
\SYMMLQ{} is reliable on compatible systems
even if $A$ is ill-conditioned or singular,
while on (singular) incompatible problems its iterates $x_k$
diverge to a multiple of a nullvector of $A$
\cite[Proposition 2.15]{C06} and \cite[Lemma 2.17]{C06}.
\MINRES{} seems more desirable to users because its
residual norms are monotonically decreasing.
%It is reliable on
%compatible systems as long as $A$ is not too ill-conditioned.
On singular compatible systems, \MINRES{} returns $x^\dagger$
(see Theorem \ref{theorem-singular-compatible}).
On singular incompatible systems, \MINRES{} is reliable if terminated with a
suitable stopping rule involving $\norm{Ar_k}$
(see Lemma \ref{minreslemma2}),
%SC: moved up the phrase about stopping rule
%\MINRES{} should also be reliable on singular incompatible systems
%if $A$ is not too ill-conditioned.
but the solution will not be $x^\dagger$.

Here we develop a new solver of this type named \MINRESQLP{}
\cite{C06}.  The aim is to deal reliably with compatible or
incompatible systems and to return the \emph{unique} solution of
\eqref{eqn4b}.  We give theoretical reasons why \MINRESQLP{} improves
the accuracy of \MINRES{} on ill-conditioned systems, and illustrate
with numerical examples.

Incompatible symmetric systems could arise from
discretized semidefinite Neumann boundary value problems \cite[section~4]{K88},
and from any other singular systems involving measurement errors in $b$.
Another potential application is large symmetric indefinite low-rank
Toeplitz LS problems as described in \cite[section~4.1]{GTVV96}.

\subsection{Notation}  \label{sec:notation}

The letters $i$, $j$, $k$ denote integer indices, $c$ and $s$
cosine and sine of some angle $\theta$, $e_k$
the $k$th unit vector, $e$
a vector of all ones, and other lower-case letters such as $b$, $u$,
and $x$ (possibly with integer subscripts) denote \textit{column}
vectors.
Upper-case letters $A$, $T_k$, $V_k$, \dots{} denote matrices, and
$I_k$ is the identity matrix of order $k$.  Lower-case Greek letters
denote scalars; in particular, $\varepsilon \approx 10^{-16}$ denotes
the floating-point precision.
If a quantity $\delta_k$ is modified one or more times, we denote its
values by $\delta_k$, $\delta_k^{(2)}$, $\delta_k^{(3)}$, \dots.
The symbol $\norm{\,\cdot\,}$ denotes the $2$-norm of a
vector or matrix.
%We use $\diag(v)$ to denote a diagonal matrix with
%elements of a vector $v$ on the diagonal,
%$\mathcal{R}(\cdot)$ the range and
%$\mathcal{N}(\cdot)$ the nullspace of a matrix,
%while $\mathcal{K}_k(A,b)$ is the $k$th Krylov subspace of $A$ and $b$.
%SC: made reference to (1.1)
For an incompatible system, $Ax\approx b$ is shorthand
for the LS problem \hbox{$\min_x\norm{Ax-b}$}.

%SC: defined positive definite matrix
%A symmetric matrix $A$ is positive definite if all its eigenvalues are
%positive and we write $A \succ 0$.

\subsection{Overview}

In sections \ref{sec:Lanczos}--\ref{sec:QLPreview}
we briefly review the Lanczos process, \MINRES{}, and QLP decomposition
before introducing \MINRESQLP{} in section \ref{sec:MINRESQLP}.
We derive norm estimates in section \ref{sec:QLPstop} and
preconditioned \MINRESQLP{} in section \ref{sec:pminres}.
Numerical experiments are described in section \ref{sec:numerical}.

\section{The Lanczos process}  \label{sec:Lanczos}

%SC: removed "reliable" and added "in terms of numerical stability" to Chris's form
Given $A$ and $b$, the Lanczos process \cite{L50}
computes vectors $v_k$ and tridiagonal matrices $\underline{T_k}$
according to $v_0 \equiv 0$, $\beta_1 v_1=b$, and then%
\footnote{Numerically, $p_k = Av_k-\beta_kv_{k-1}$, $\alpha_k = v_k\T p_k$,
$\beta_{k+1}v_{k+1} = p_k-\alpha_kv_k$ is slightly better \cite{P76}.}
\begin{equation*}
    p_k = Av_k, \qquad \alpha_k = v_k\T p_k,
                \qquad \beta_{k+1}v_{k+1} = p_k-\alpha_kv_k-\beta_kv_{k-1}
\end{equation*}
for $k=1,2,\dots,\ell$, where we choose $\beta_k > 0$ to give
$\norm{v_k}=1$.
In matrix form,
\begin{equation}  \label{eq:avk}
  AV_k \!=\! V_{k+1}\underTk,\quad
  \underTk \!\equiv\! \mbox{\footnotesize
    $\bmat{\alpha_{1} & \beta_{2}
        \\ \beta_{2}  & \alpha_{2} & \ddots
        \\            & \ddots & \ddots & \beta_k
        \\            &  & \beta_k & \alpha_k
        \\            &  &  & \beta_{k+1}}
    $}
    \!\equiv\!
    \bmat{T_k \\ \beta_{k+1}e_k^T},
    \quad V_k \!\equiv\! \bmat{v_1 & \!\cdots\! & v_k}.
\end{equation}
In exact arithmetic, the columns of $V_k$ are orthonormal and the
process stops with $k = \ell$ and $\beta_{\ell+1}=0$ for some
$\ell \le n$, and then $AV_\ell = V_\ell T_\ell$.
For derivation purposes we assume that this happens,
%that $V_k^T V_k\drop = I$ for all $k$.  We also allow $\beta_{k+1}=0$
%for completeness, though that case may be skipped on a first reading.)
though in practice it is unlikely unless $V_k$ is reorthogonalized for each $k$.
%In practice it is likely that $\beta_k > 0$ for all $k$, but
In any case, \eqref{eq:avk} holds to machine precision and the computed vectors
satisfy $\norm{V_k}_1 \approx 1$ (even if $k \gg n$).

%SC: \pagebreak

%\begin{assumption}[Orthogonality of $V_k$]  \label{assume:orthog}
%    For example, in
%  deriving estimates of certain norms, if $q_k = \alpha v_{k-1} +
%  \beta v_k$ for some scalars $\alpha$ and $\beta$, we would assume
%  that $\norm{q_k} = \norm{[\alpha \ \; \beta]}$.  Even though their
%  relative accuracies may diminish, the resulting norm estimates help
%  determine a safe point of termination.
%\end{assumption}

\subsection{Properties of the Lanczos process} \label{sec:Lanproperties}

The $k$th Krylov subspace generated by $A$ and $b$ is defined to be
$\mathcal{K}_k(A,b) = \Span \{b, Ab, A^2b, \dots, A^{k-1}b\} =
\Span(V_k)$.  The following properties should be kept in mind:

\begin{enumerate}

\item If $A$ is changed to $A-\sigma I$ for some scalar shift
  $\sigma$, $T_k$ becomes $T_k - \sigma I$ and $V_k$ is unaltered,
  showing that singular systems are commonplace.  Shifted problems
  appear in inverse iteration or Rayleigh quotient iteration.

%\item If $A$ is positive definite, so is $T_k$ for all $k$.

\item $\underTk$ has full column rank $k$ for all $k < \ell$.

\item If $A$ is indefinite, some $T_k$ might be singular for $k < \ell$, but then by the
  Sturm sequence property (see \cite{GV}), $T_k$ has exactly one zero
  eigenvalue and the strict interlacing property implies that $T_{k
    \pm 1}$ are nonsingular.  Hence $T_k$ cannot be singular twice in
  a row (whether $A$ is singular or not).

%\item If $A$ is nonsingular, $AV_\ell = V_\ell T_\ell$ shows that
%  $T_\ell$ is nonsingular.  If $A$ is singular, the rank of $T_\ell$
%  depends on $b$.

\item $T_\ell$ is nonsingular if and only if $b \in \range(A)$.
  (See appendix \ref{appendixA}.)
\end{enumerate}

\section{\MINRES}  \label{sec:MINRESreview}

Algorithm \MINRES{} \cite{PS75} is a natural way of using the Lanczos process
%\eqref{eq:avk}
to solve $Ax = b$ or $\min_x \norm{Ax-b}$.  For $k < \ell$, if $x_k =
V_ky_k$ for some vector $y_k$, the associated residual is
\begin{equation}  \label{eqn:rk}
   r_k \equiv b-Ax_k = b - AV_k y_k
       = \beta_1 v_1 - V_{k+1} \underTk y_k
       = V_{k+1} (\beta_1 e_1 - \underTk y_k).
\end{equation}
To make $r_k$ small, it is clear that $\beta_1 e_1 - \underTk y_k$
should be small.  At this iteration $k$, \MINRES{} minimizes the residual
subject to $x_k\in\mathcal{K}_k(A,b)$ by choosing
\begin{equation}  \label{eqn:LSsubprob}
  y_k = \arg\min_{y \in \mathbb{R}^k}
        \norm{\underTk y - \beta_1 e_1}.
\end{equation}
This subproblem is processed by the expanding QR factorization:
$Q_0 \equiv 1$ and
\begin{equation}  \label{QRfac}
  Q_{k,k+1}\!\equiv\!
  \left[\begin{smallmatrix}
           I_{k-1}&&
        \\        & c_k &   \!\!\phantom-s_k
        \\        & s_k &   \!-c_k
        \end{smallmatrix}\right]\!\!,
                 \quad
  Q_k \!\equiv\! Q_{k,k+1}
                 \bmat{Q_{k-1} \\ & \!1}\!,
                 \quad
                 Q_k \bmat{\underTk & \beta_1 e_1}
      \!=\!      \bmat{ R_k & t_k \\ 0 & \phi_k}\!,
\end{equation}
where $c_k$ and $s_k$ form the Householder reflector $Q_{k,k+1}$ that
annihilates $\beta_{k+1}$ in $\underTk$ to give
upper-tridiagonal $R_k$, with $R_k$ and $t_k$ being unaltered in later
iterations.

% When $\beta_{k+1}>0$, $R_k$ is nonsingular
When $k < \ell$, % $\underTk$ and $R_k$ have full rank and
the unique solution of
\eqref{eqn:LSsubprob} satisfies $R_k y_k = t_k$.  Instead of solving
for $y_k$, \MINRES{} solves $R_k\T D_k\T = V_k\T$ by forward
substitution, obtaining the last column $d_k$ of $D_k$ at iteration
$k$.  At the same time, it updates $x_k$ via
$x_0 \equiv 0$ and
\begin{equation}  \label{minresxk}
    x_k = V_k y_k = D_k R_k y_k = D_k t_k
  = x_{k-1} + \tau_k d_k,\quad \tau_k\equiv e_k^Tt_k.
\end{equation}

When $k=\ell$, we can form $T_\ell$ but nothing else expands.
In place of \eqref{eqn:rk} and \eqref{QRfac} we have
\(
   r_\ell = V_\ell (\beta_1 e_1 - T_\ell y_\ell)
\)
and
\(
   Q_{\ell-1} \bmat{T_\ell & \beta_1 e_1} = \bmat{R_\ell & t_\ell}
\)
and it is natural to choose $y_\ell$ from the subproblem
\begin{equation}  \label{eqn:LSsubprob-ell}
   \min \norm{T_\ell y_\ell - \beta_1 e_1}
   \quad \equiv \quad
   \min \norm{R_\ell y_\ell - t_\ell}.
\end{equation}
There are two cases to consider:
\begin{enumerate}
\item If $T_\ell$ is nonsingular, $R_\ell y_\ell = t_\ell$ has a
  unique solution.  Since $AV_\ell y_\ell = V_\ell T_\ell y_\ell =
  b$, the problem is solved by $x_\ell = V_\ell y_\ell $ with residual
  $r_\ell=0$ (the system is compatible, even if $A$ is singular).
  Theorem \ref{theorem-singular-compatible} proves that $x_\ell =
  x^\dagger$.

\item If $T_\ell$ is singular, $A$ and $R_\ell$ are singular
  ($R_{\ell\ell}=0$) and both $Ax = b$ and $R_\ell y_\ell = t_\ell$ are
  incompatible.  This case was not handled by \MINRES{} in \cite{PS75}.
  Theorem \ref{theorem-singular-incompatible} proves that the \MINRES{}
  point $x_{\ell-1}$ is a least-squares solution (but not necessarily
  $x^\dagger$).  Theorem \ref{theorem-MINRES-QLP} proves that the
  \MINRESQLP{} point $x_\ell = V_\ell y_\ell^\dagger = x^\dagger$, where
  $y_\ell^\dagger$ is the min-length solution of
  \eqref{eqn:LSsubprob-ell}.
\end{enumerate}

\subsection{Further details of \MINRES}   \label{sec:MINRESdetails}

To describe \MINRESQLP{} thoroughly, we need further
details of the \MINRES{} QR factorization \eqref{QRfac}.
For $1 \le k < \ell$,
\begin{equation}  \label{QRfac2}
  \bmat{\,R_k\, \\ 0} = \mbox{\small
$\bmat{
      \gamma_1 & \delta_2 & \epsilon_3
\\                   & \gamma_2^{(2)} & \delta_3^{(2)} & \ddots
\\                   &                & \ddots         & \ddots & \epsilon_k
\\                   &                &                & \ddots & \delta_k^{(2)}
\\                   &                &                &        & \gamma_k^{(2)}
\\                   &                &                &        & 0}$},
  \quad
  \bmat{t_k \\ \phi_k} \equiv
  \bmat{\tau_1\\ \tau_2\\ \vdots\\ \vdots\\ \tau_k\\ \phi_k}
  = \beta_1
  \bmat{c_1 \\ s_1c_2 \\ \vdots \\ \vdots
     \\ s_1 \cdots s_{k-1}c_k
     \\ s_1 \cdots s_{k-1}s_k}
\end{equation}
(where the superscripts are defined in section~\ref{sec:notation}).
With $\phi_0 \equiv \beta_1>0$,
the full action of $Q_{k,k+1}$ in \eqref{QRfac}, including its effect
on later columns of $T_j$, $k < j \le \ell$, is described~by
\begin{equation}  \label{min7}
  \bmat{   c_k & \!\!\!\phantom-s_k
        \\ s_k & \!\!-c_k}
  \bmat{\begin{matrix}
           \gamma_k & \delta_{k+1} & 0
        \\ \beta_{k+1}    & \alpha_{k+1}       & \beta_{k+2}
        \end{matrix}
        & \biggm| &
        \begin{matrix} \phi_{k-1} \\ 0 \end{matrix}
       }
=
  \bmat{\begin{matrix}
           \gamma_k^{(2)} & \delta_{k+1}^{(2)} & \epsilon_{k+2}
        \\ 0              & \gamma_{k+1} & \delta_{k+2}
        \end{matrix}
        & \biggm| &
        \begin{matrix} \tau_k \\ \phi_k \end{matrix}
       },
\end{equation}
where $s_k = \beta_{k+1}/\norm{[\gamma_k \ \; \beta_{k+1}]} > 0$,
giving $\gamma_1$, $\gamma_k^{(2)} > 0$ with $R_j$ nonsingular
for each $j \le k < \ell$.
Thus the $d_j$ in \eqref{minresxk} can be found from
\begin{align}
   R_k\T D_k\T = V_k\T\ :\ &
   \left\{
     \begin{array}{l}
       d_1 = v_1/\gamma_1,\quad
       d_2 = (v_2-\delta_2d_1)/\gamma_2^{(2)}, \\
       d_j = ({v_j - \delta_j^{(2)} d_{j-1}
              - \epsilon_j d_{j-2}}) / {\gamma_j^{(2)}},
       \quad j=3,\ldots,k.
     \end{array}
   \right.
   \label{eq:rdeqv}
\end{align}
Also, $\tau_k = \phi_{k-1} c_k$ and $\phi_k = \phi_{k-1} s_k > 0$.
Hence from \eqref{eqn:rk}--\eqref{QRfac},
\begin{align}
   \label{eq:normrk}
   \norm{r_k} = \norm{\underTk y_k - \beta_1 e_1} = \phi_k
   \quad\Rightarrow\quad \norm{r_k} = \norm{r_{k-1}} s_k,
\end{align}
which is nonincreasing and tending to zero if $Ax = b$ is
compatible.

\begin{comment}
\begin{remark}
\label{rem:stop1}
If $\gamma_k = 0$ then $T_k$ is singular.  If $k < \ell$ %$\beta_{k+1} > 0$
we have $s_k=1$ and $\norm{r_k} = \norm{r_{k-1}}$ (not a strict
decrease), but this cannot happen twice in a row.
%If $\gamma_k = 0$ and $\beta_{k+1}=0$, then $Q_{k-1,k}$
%terminates the decomposition, the equivalent of \eqref{QRfac2} has
%$\smash{\gamma_k^{(2)}} = 0$, and the final residual stays as $r_{k-1}$ with
%$\norm{r_{k-1}} = \phi_{k-1} = \beta_1 s_1 \cdots s_{k-1} \ne 0$, see
%Remark~\ref{rem:ns}, but $y_k$ is not unique in \eqref{eqn:LSsubprob}.
If $k=\ell-1$ and $\gamma_{k+1} = \gamma_\ell = 0$ in \eqref{min7},
then the final point and residual stay as $x_{\ell-1}$ and $r_{\ell-1}$ with
$\norm{r_{\ell-1}} = \phi_{\ell-1} = \beta_1 s_1 \cdots s_{\ell-1} > 0$.
%, see Remark~\ref{rem:ns},
%but $y_\ell$ is not unique in \eqref{eqn:LSsubprob}.
\end{remark}
\end{comment}

\begin{remark} \label{rem:stop1}
  If $k < \ell $ and $T_k$ is singular, we have $\gamma_k = 0$, $s_k =
  1$, and $\norm{r_k} = \norm{r_{k-1}}$ (not a strict decrease), but
  this cannot happen twice in a row (cf.~section~\ref{sec:Lanproperties}).
\end{remark}

\begin{remark} \label{rem:stop2}
  If $T_\ell$ is singular, \MINRES{} sets the last element of
  $y_\ell$ to be zero.  The final point and residual stay as
  $x_{\ell-1}$ and $r_{\ell-1}$ with $\norm{r_{\ell-1}} =
  \phi_{\ell-1} = \beta_1 s_1 \cdots s_{\ell-1} > 0$.
\end{remark}

\subsection{Compatible systems}

The following theorem assures us that \MINRES{} is a useful solver for
compatible linear systems even if $A$ is singular.

\begin{theorem}[{\cite[Theorem 2.25]{C06}}]  \label{theorem-singular-compatible}
  If $b \in \range(A)$, the final \MINRES{} point $x_\ell$ is the
  minimum-length solution of $Ax=b$ (and $r_\ell = b - Ax_\ell = 0$).
\end{theorem}

\begin{proof}
%This follows from Remark~\ref{rem:ns}.
If $b \in \range(A)$, the Lanczos process gives $AV_\ell = V_\ell T_\ell$
 with nonsingular $T_\ell$, and \MINRES{} terminates with $Ax_\ell=b$ and
 $x_\ell = V_\ell y_\ell = Aq$, where $q = V_\ell T_\ell\inv y_\ell$.
 If some other point $\xhat$ satisfies $A\xhat=b$, let $p = \xhat - x_\ell$.
 We have $Ap=0$ and $x_\ell^T p = q\T Ap = 0$.
 Hence $\norm{\xhat}^2 = \norm{x_\ell + p}^2 = \norm{x_\ell}^2 + 2
 x_\ell^T p + \norm{p}^2 \ge \norm{x_\ell}^2$.
\end{proof}

\subsection{Incompatible systems}

For a singular LS problem $Ax \approx b$, the optimal residual vector
$\rhat$ is unique, but infinitely many solutions $x$ give that residual.
In the following example, \MINRES{} finds a least-squares solution
(with optimal residual) but not the minimum-length solution.

\begin{example} \label{minresCounterEg}
  Let $A = \diag(1,1,0)$ and $b=e$.
  The minimum-length solution to $Ax \approx b$  is
  $x^\dagger = [ 1 \ 1 \ 0 ]\T$ with residual
  $\rhat = b - Ax^\dagger = e_3$ and $A\rhat = 0$.
  \MINRES{} returns the solution $x^\sharp = e$
  (with residual $r^\sharp = b - Ax^\sharp= e_3 = \rhat$ and $Ar^\sharp = 0$).
\end{example}

\begin{theorem}[{\cite[Theorem 2.27]{C06}}]  \label{theorem-singular-incompatible}
  If $b \notin \range(A)$, then
  $\norm{Ar_{\ell-1}} = 0$ and the \MINRES{} $x_{\ell-1}$ is an
  LS solution (but not necessarily $x^\dagger$).
\end{theorem}

\begin{proof}
Since $b \notin \range(A)$, $T_\ell$ is singular and $R_{\ell\ell} = \gamma_\ell = 0$.
By Lemma~\ref{minreslemma2} below, $A(Ax_{\ell-1}-b) = -Ar_{\ell-1} = -\norm{r_{\ell-1}} \gamma_\ell v_\ell = 0$.
Thus $x_{\ell-1}$ is an LS solution.
\end{proof}

\subsection{Norm estimates in \MINRES}  \label{sec:MINRES-norms}

For incompatible systems, $r_k$ \eqref{eqn:rk} will never be
zero.  However, all LS solutions satisfy $A^2x = Ab$, so that $Ar =
0$.  We therefore need a new stopping condition based on the size of
$\norm{Ar_k}$.  In applications requiring nullvectors, $\norm{Ax_k}$
is also useful.  We present efficient recurrence relations for
$\norm{Ar_k}$ and $\norm{Ax_k}$ in the following Lemma,
which was not considered in the framework of \MINRES{} when
it was originally designed for nonsingular systems \cite{PS75}.

\begin{lemma}[$Ar_k$, $\norm{A r_k}$, $\norm{Ax_k}$ for \MINRES{}] \label{minreslemma2}
If $k < \ell$,
\begin{align*}
   A r_k &= \norm{r_k} \left( \gamma_{k+1} v_{k+1} +
            \delta_{k+2} v_{k+2}\right)
   \quad
   \qquad(\mbox{where $\delta_{k+2} v_{k+2}=0$ if $k = \ell-1$}), % $\beta_{k+2}=0$}),
\\
   \psi_k^2 &\equiv \norm{Ar_k}^2 = \norm{r_k}^2
   \left( [\gamma_{k+1}]^2 + [\delta_{k+2}]^2 \right)
   \qquad(\mbox{where $\delta_{k+2}=0$        if $k = \ell-1$}), % $\beta_{k+2}=0$}),
\\
   \omega_k^2 &\equiv \norm{Ax_k}^2 = \omega_{k-1}^2 + \tau_k^2,
   \quad \omega_0\equiv 0.
\end{align*}
\end{lemma}

%\vspace{-2\medskipamount}

\begin{proof}
For $k < \ell$, $R_k$ is nonsingular.
From \eqref{eqn:rk}--\eqref{minresxk} with $R_ky_k=t_k$ we have
\begin{align}
   r_k %&= b -A x_k = V_{k+1} ( \beta_1 e_1 - \underTk y_k )     \nonumber
       &= V_{k+1} Q_k^T \pmat{ \bmat{t_k \\ \phi_k} - \bmat{R_k \\ 0}  y_k}
          = \phi_k V_{k+1} Q_k^T e_{k+1},                        \label{rk7}
\\Ar_k &= \phi_k V_{k+2} \underTkp Q_k\T e_{k+1},                \nonumber
\\ Q_k \underTkp\T
       &= Q_k \bmat{T_{k+1} & \beta_{k+2} e_{k+1}}
        = Q_k \bmat{     T_k           & \beta_{k+1} e_k  & 0
                 \\ \beta_{k+1} e_k^T  & \alpha_{k+1}     & \beta_{k+2}},
                                                              \nonumber
\\ e_{k+1}^TQ_k \underTkp \T & =
                     \bmat{0 & \gamma_{k+1} & \delta_{k+2}},    \nonumber
\end{align}
see \eqref{min7}, where $AV_{k+1}=V_{k+1}T_{k+1}$ and we take
$\delta_{k+2}=0$ if $k = \ell-1$, so
\begin{align*}
   Ar_k &= \phi_k V_{k+2} \bmat{0 & \gamma_{k+1} & \delta_{k+2}}\T
         = \phi_k \left( \gamma_{k+1} v_{k+1} + \delta_{k+2} v_{k+2}
                 \right),
\\ \psi_k^2 &\equiv \norm{Ar_k}^2 = \norm{r_k}^2
      \left( [\gamma_{k+1}]^2 + [\delta_{k+2}]^2 \right).
\end{align*}
For the recurrence relations of $Ax_k$ and its norm, we have
\begin{align*}
   Ax_k &= AV_k y_k = V_{k+1} \underTk y_k
                    = V_{k+1} Q_k^T  \bmat{R_k \\ 0} y_k
                    = V_{k+1} Q_k^T  \bmat{t_k \\ 0},
\\ \omega_k^2 &\equiv \norm{Ax_k}^2 = \norm{t_k}^2
                                    = \norm{t_{k-1}}^2 + \tau_k^2
                                    = \omega_{k-1}^2   + \tau_k^2.
\end{align*}
\end{proof}

Note that even using finite precision the expression for $\psi_k^2$ is
extremely accurate for the versions of the Lanczos algorithm given in
section~\ref{sec:Lanczos}, since (taking $\norm{v_j}=1$ with negligible
error), $\norm{Ar_k}^2 = \phi_k^2( [\gamma_{k+1}]^2
+2\gamma_{k+1}\delta_{k+2} v_{k+1}^T v_{k+2} +
[\delta_{k+2}]^2)$, where from \eqref{min7}
$|\delta_{k+2}|\leq \beta_{k+2}$, while from \cite[(18)]{P76}
$\beta_{k+2}|v_{k+1}^T v_{k+2}|\leq O(\varepsilon)\norm{A}$, and with
$|\gamma_{k+1}|\leq \norm{A}$, see \cite[(19)]{P76}, we see that
$|\gamma_{k+1}\delta_{k+2} v_{k+1}^T v_{k+2}| \leq
O(\varepsilon)\norm{A}^2$.

Typically $\norm{Ar_k}$ is not monotonic, while clearly $\norm{r_k}$
and $\norm{Ax_k}$ \emph{are} monotonic.
In the eigensystem $A =U \Lambda U\T$, let $U = \bmat{U_1 \!&\! U_2}$,
where the eigenvectors $U_1$ correspond to nonzero eigenvalues.  Then
$P_A \equiv U_1 U_1\T$
and $P^{\perp}_A \equiv U_2 U_2\T$
are orthogonal projectors \cite{TB} onto the range and nullspace of
$A$.  For general linear LS problems, Chang et al.\
\cite{CPP09} characterize the dynamics of $\norm{r_k}$ and
$\norm{A^Tr_k}$ in three phases defined in terms of the ratios among
$\norm{r_k}$, $\norm{P_A r_k}$, and $\norm{P^{\perp}_A r_k}$, and
propose two new stopping criteria for iterative solvers.
The expositions \cite{AG08, JT10} show that these estimates are cheaply computable in
\CGLS{} and \LSQR{} \cite{PS82a,PS82b}.

\subsection{Effects of rounding errors in \MINRES}

\MINRES{} should stop if $R_k$ is singular (which theoretically
implies $k=\ell$ and $A$ is singular).  Singularity was not discussed
by Paige and Saunders \cite{PS75}, but they did raise the question: Is
\MINRES{} stable when $R_k$ is ill-conditioned?  Their concern was
that $\norm{D_k}$ could be large in \eqref{eq:rdeqv}, and there could
then be cancellation in forming $x_{k-1} + \tau_k d_k$ in
\eqref{minresxk}.

Sleijpen, Van der Vorst, and Modersitzki \cite{SVM00} analyzed the
effects of rounding errors in \MINRES{} and reported examples of
apparent failure with a matrix of the form $A = QDQ\T$, where $D$ is
an ill-conditioned diagonal matrix and $Q$ involves a single plane
rotation.  We were unable to reproduce \MINRES{}'s performance on the
two examples defined in Figure~4 of their paper, but we modified the
examples by using an $n \times n$ Householder transformation for $Q$,
and then observed similar difficulties with \MINRES{}---see
Figure~\ref{figDPtestSing3_DP}.
The recurred residual norm $\phi^M_k$ is a good approximation to the
directly computed $\norm{r^M_k}$ until the last few iterations.  The
recurred norms $\phi^M_k$ then keep decreasing but the directly
computed norms $\norm{r^M_k}$ become stagnant or even increase (see
the lower subplots in Figure~\ref{figDPtestSing3_DP}).

\begin{remark}
  Note that we do want $\phi_k$ to keep decreasing on compatible
  systems, so that the test $\phi_k \leq \mathit{tol} (\norm{A}
  \norm{x_k}+\norm{b})$ with $\mathit{tol} \ge \varepsilon$ will
  eventually be satisfied even if the computed $\norm{r_k}$ is no
  longer as small as $\phi_k$.
\end{remark}

The analysis in \cite{SVM00} focuses on the rounding errors involved
in the $n$ lower triangular solves $R_k\T D_k\T = V_k\T$ (one solve
for each row of $D_k$), compared to the single upper triangular solve
$R_k y_k = t_k$ (followed by $x_k = V_k y_k$) that would be possible
at the final $k$ if all of $V_k$ were stored as in \GMRES{}
\cite{SS86}.  We shall see that a key feature of \MINRESQLP{} is that
a single lower triangular solve suffices with no need to store $V_k$,
much the same as in \SYMMLQ.

\section{Orthogonal decompositions for singular matrices}  \label{sec:QLPreview}

In 1999 Stewart proposed the \emph{pivoted QLP decomposition}
\cite{S99}, which is equivalent to two consecutive QR factorizations
with column interchanges, first on $A$, then on $R^T$:
\begin{equation}  \label{qlpeqn1}
   Q_R A \Pi_R = \bmat{ R & S \\ 0&0 }, \qquad
   Q_L \bmat{R^T & 0 \\ S^T & 0} \Pi_L = \bmat{\hat{R} & 0 \\ 0&0},
\end{equation}
giving nonnegative diagonal elements,
where $\Pi_R$ and $\Pi_L$ are permutations chosen to maximize the
next diagonal element of $R$ and $\hat{R}$ at each stage.  This gives
$A = QLP$, where
\begin{equation*}
  Q = Q_R^T \Pi_L, \qquad
  L = \bmat{\hat{R}^T & 0 \\ 0&0}, \qquad
  P =  Q_L \Pi_R^T,
\label{qlpeqn2}
\end{equation*}
with $Q$ and $P$ orthogonal.  Stewart demonstrated that the diagonals
of $L$ (the \emph{$L$-values}) give better singular-value estimates
than the diagonals of $R$ (the \emph{$R$-values}), and the accuracy
is particularly good for the extreme singular values $\sigma_1$ and
$\sigma_n$:
\begin{equation}  \label{qlpeqn1a}
  R_{ii} \approx \sigma_i, \quad L_{ii} \approx \sigma_i, \quad
  \sigma_1 \ge  \max_i L_{ii} \ge \max_i R_{ii}, \quad
  \min_i R_{ii} \ge \min_i L_{ii} \ge \sigma_n.\!\!
\end{equation}
The first permutation $\Pi_R$ in pivoted QLP is important.  The main
purpose of the second permutation $\Pi_L$ is to ensure that the
$L$-values present themselves in decreasing order, which is not always
necessary.  If $\Pi_R = \Pi_L = I$, it is simply called the \emph{QLP
  decomposition}.

\begin{comment}
\begin{remark} \label{rem:ns}
  In \eqref{qlpeqn3a}, $\beta_{k+1}>0$ implies $R_k$ and $L_k$ are
  nonsingular (and $k<\ell$).
  %Suppose $\beta_{k+1}=0$.  In
  %section~\ref{sec:Lanczos}, $b \in \mathcal{R}(A)$ if and only if
  %$v_1,\ldots, v_k \perp \mathcal{N}(A)$,
  %so $b \in \mathcal{R}(A)$ if and only if $T_k$ is nonsingular (since
  %$AV_k=V_kT_k$).
  If $T_\ell$ is nonsingular then $T_\ell y = \beta_1 e_1$
  has a unique solution and $AV_ky_k = V_kT_ky_k = b$, where
  $x_k = V_ky_k\perp \Null(A)$, and $r_k = 0$; otherwise $b \not\in
  \range(A)$, $T_k$ is singular, and the residual is nonzero.
\end{remark}
\end{comment}

\section{\MINRESQLP}  \label{sec:MINRESQLP}

We now develop \MINRESQLP{} for solving ill-conditioned or singular
symmetric systems $Ax \approx b$.  The Lanczos framework is the same
as in \MINRES, but we handle $T_\ell$ in \eqref{eqn:LSsubprob-ell} with extra care
when it is rank-deficient.
%      or $T_k$ is ill-conditioned.
%%% Can't say this when lots of T_k could be singular.
%%% We could say
%      or earlier $\underTk$ are ill-conditioned.
%%% but it confuses things.
In this case, the normal approach
to solving \eqref{eqn:LSsubprob-ell} is via a QLP decomposition of $T_\ell$ to
obtain the (unique) minimum-length solution $y_\ell$ \cite{S99,GV}.
Thus in \MINRESQLP{} we use a QLP decomposition of $\underTk$ in
subproblem \eqref{eqn:LSsubprob} for all $k \le \ell$.  This is the
\MINRES{} QR \eqref{QRfac} followed by an LQ factorization of $R_k$:
\begin{equation}  \label{qlpeqn3a}
  Q_k \underTk =  \bmat{R_k \\ 0}, \qquad
  R_k P_k = L_k,  \qquad \textrm{so that}\quad
  Q_k \underTk P_k =  \bmat{L_k \\ 0},
\end{equation}
where $Q_k$ and $P_k$ are orthogonal, $R_k$ is upper tridiagonal and
$L_k$ is lower tridiagonal.  When $k < \ell$, $R_k$ and $L_k$ are
nonsingular.  \MINRESQLP{} obtains the same solution as \MINRES{}, but
by a different process (and with different rounding errors).  Defining
$u$ by $y=P_ku$, we see from \eqref{QRfac} that
\[
   Q_k (\underTk y - \beta_1 e_1)
     = \bmat{ L_k  \\ 0}u-\bmat{ t_k \\ \phi_k},
\]
and \eqref{eqn:LSsubprob} is solved by $L_ku_k = t_k$ and $y_k = P_ku_k$.
The \MINRESQLP{} estimate of $x$ is therefore
\(
  x_k = V_ky_k = V_kP_ku_k = W_ku_k,
\)
with theoretically orthonormal $W_k\equiv V_kP_k$.
% Thus, there should be the same numerical advantage that \SYMMLQ{} has over \MINRES.

We will see that only the last three columns of $V_k$ are needed
to update $x_k$.

\begin{comment}
\subsection{The \MINRESQLP{} subproblem}

If $k < \ell$, $\underTk$ has full column rank and $y_k$ solves
the same subproblem $\min \norm{\underTk y - \beta_1 e_1}$
\eqref{eqn:LSsubprob} as \MINRES{}.
If $A$ is singular and $b \not\in \range(A)$,
$T_\ell$ is singular and we choose $y_\ell$ to solve the
subproblem
\begin{equation}  \label{qlpsubproblemk}
  \min \norm{y}
         \quad\text{s.t.}\quad
         y \in \arg\min_{y \in \mathbb{R}^\ell}
            \norm{T_\ell y- \beta_1 e_1}.
\end{equation}
For all $k \le \ell$ the solutions $y_k$ define
$x_k = V_k y_k$ % \in \mathcal{K}_k (A,b)$
as the $k$th approximation to $x$.
As usual, $y_k$ is not actually computed because all its
elements change when $k$ increases.
\end{comment}

\subsection{The QLP factorization of $\underTk$}

The QLP decomposition of each $\underTk$ must be without
permutations in order to ensure inexpensive updating of the factors as $k$
increases.  Our experience is that the desired rank-revealing
properties \eqref{qlpeqn1a} tend to be retained, perhaps because of
the tridiagonal structure of $\underTk$ and the convergence
properties of the underlying Lanczos process.

For $k < \ell$, the QLP decomposition of $\underTk$ \eqref{qlpeqn3a}
%\begin{equation}  \label{qlpeqn3}
%  Q_k \underTk = \bmat{R_k \\ 0},
%  \qquad
%  R_k P_k = L_k,
%  \qquad
%  Q_k \underTk P_k = \bmat{L_k \\ 0},
%\end{equation}
gives nonsingular tridiagonal $R_k$ and $L_k$.  As in \MINRES,
$Q_k$ is a product of Householder reflectors, see \eqref{QRfac} and
\eqref{min7}, while $P_k$ involves a product of \emph{pairs} of
essentially $2\times 2$
reflectors:
\begin{equation*}
  Q_k = Q_{k,k+1}\ \cdots \ Q_{3,4} \ \ Q_{2,3} \ \ Q_{1,2},\qquad
  P_k  =  P_{1,2} \ \ P_{1,3} P_{2,3}
  \ \cdots \ \ P_{k-2,k} P_{k-1,k}.
\end{equation*}
For \MINRESQLP{} to be efficient, in the $k$th iteration ($k\ge 3$)
the application of the left reflector $Q_{k,k+1}$ is followed
immediately by the right reflectors $P_{k-2,k}, P_{k-1,k}$, so that
only the last $2 \times 2$ principal submatrix of %% ?? $\underline{L_k}$
the transformed $\underTk$ will be changed in future
iterations.
These ideas can be understood more easily from
Figure~\ref{QLPfig}
and the following compact form, which represents the actions of
right reflectors on $\underTk$ (additional to $Q_{k,k+1}$ \eqref{min7}):
\begin{align}
\label{qlpRightRef}
 &\hspace*{13pt}
  \bmat{\gamma_{k-2}^{(5)} &                    & \epsilon_k
     \\ \vartheta_{k-1}    & \gamma_{k-1}^{(4)} & \delta_k^{(2)}
     \\                    &                    & \gamma_{k}^{(2)}
       }
  \bmat{c_{k2} &   & \!\!\!\phantom-s_{k2}
     \\        & 1 &
     \\ s_{k2} &   & \!\!-c_{k2}
       }
  \bmat{ 1
     \\    & c_{k3} & \!\!\!\phantom-s_{k3}
     \\    & s_{k3} & \!\!-c_{k3}
       }
       \nonumber
\\ &=
  \bmat{\gamma_{k-2}^{(6)}
     \\ \vartheta_{k-1}^{(2)} & \gamma_{k-1}^{(4)}  & \delta_k^{(3)}
     \\ \eta_{k}              &                     & \gamma_{k}^{(3)}
       }
  \bmat{ 1
     \\    & c_{k3} & \!\!\!\phantom-s_{k3}
     \\    & s_{k3} & \!\!-c_{k3}
       }
= \bmat{\gamma_{k-2}^{(6)}
     \\ \vartheta_{k-1}^{(2)} & \gamma_{k-1}^{(5)}  &
     \\ \eta_{k}              & \vartheta_{k}       & \gamma_{k}^{(4)}
       }.
\end{align}

\begin{figure}[p]    %%% Figure 5.1
  \centering
  \includegraphics[width=\textwidth]{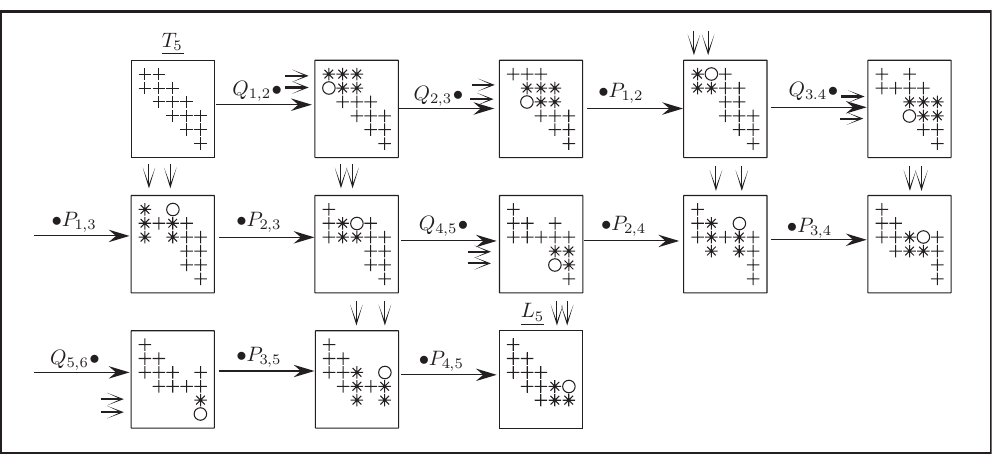}
  \caption{QLP with left and right reflectors interleaved on $\underline{T_5}$.
   This figure can be reproduced with the help of \texttt{QLPfig5.m}.}
  \label{QLPfig}
\end{figure}

%SC: Changed "Diagonals" to the following
\subsection{The diagonals of $L_k$}

Figure~\ref{Davis1177case2Eig4fig} shows the relation between the
singular values of $A$ and the diagonal elements of $R_k$ and
$L_k$ with $k=19$. This illustrates \eqref{qlpeqn1a} for matrix
ID 1177 from \cite{UFSMC} with $n=25$.

\begin{figure}[p]    %%% Figure 5.2
\centering
\hspace*{-0.1in}
\includegraphics[width=.95\textwidth]{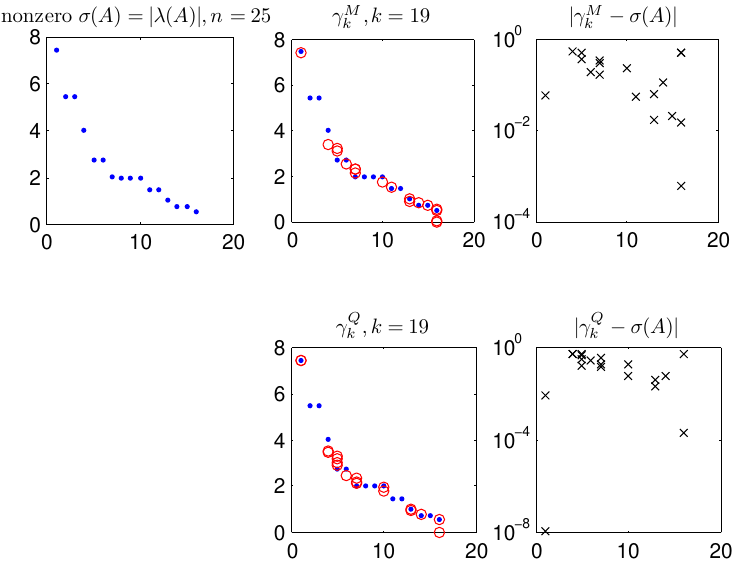}
\caption{
  \textbf{Upper left:}
  Nonzero singular values of $A$ sorted in decreasing order.
  \textbf{Upper middle and right:}
  The diagonals
  $\gamma_k^M$
  of $R_k$ (red circles)
  from MINRES are plotted as red circles above or below the
  nearest singular value of $A$.  They approximate the extreme nonzero
  singular values of $A$ well.
  \textbf{Lower:} The diagonals
  $\gamma_k^Q$
  of $L_k$ (red circles) from \MINRESQLP{} approximate the extreme
  nonzero singular values of $A$ even better.  An implication is that
  the ratio of the largest and smallest diagonals of $L_k$ provides a
  good estimate of $\kappa(A)$.  To reproduce this figure, run
  \texttt{test\_minresqlp\_fig3(2)}.}
\label{Davis1177case2Eig4fig}
\end{figure}

\subsection{Solving the subproblem}

With $y_k=P_k u_k$, subproblem \eqref{eqn:LSsubprob} becomes
%$L_k u_k = t_k$ and $y_k = P_k u_k$.
%equivalent to
\begin{equation}  \label{Lsubproblem}
  u_k = \arg \min_{u \in \mathbb{R}^k}
            \,\normm{\bmat{L_k \\ 0} u - \bmat{t_k \\ \phi_k}},
\end{equation}
where $t_k$ and $\phi_k$ are as in \eqref{QRfac} and \eqref{QRfac2}.
At the \textit{start}
of iteration $k$, the first $k\!-\!3$ elements of $u_k$, denoted by
$\mu_j$ for $j \le k\!-\!3$,
are known from previous iterations; see the 10th matrix in
Figure~\ref{QLPfig}.
The remainder depend on the rank of $L_k$.

\begin{enumerate}
\item
If $\rank({L_k})=k$ (so $k < \ell$, or $k = \ell$ and $b \in \range(A)$),
we need to solve the last three equations of $L_k u_k = t_k$:
\begin{equation}
  \bmat{\gamma_{k-2}^{(6)}
     \\ \vartheta_{k-1}^{(2)} & \gamma_{k-1}^{(5)}  &
     \\ \eta_k              & \vartheta_k       & \gamma_k^{(4)}
       }
  \bmat{\mu_{k-2}^{(3)}
     \\ \mu_{k-1}^{(2)}
     \\ \mu_k
       }
= \bmat{{\bar\tau_{k-2}}
     \\ {\bar\tau_{k-1}}
     \\ \tau_k
       }
\equiv
  \bmat{\tau_{k-2} - \eta_{k-2} \mu_{k-4}^{(4)}
                   - \vartheta_{k-2} \mu_{k-3}^{(3)}
     \\ \tau_{k-1} - \eta_{k-1} \mu_{k-3}^{(3)}
     \\ \tau_k
       }.
\end{equation}

\item
If %$\rank({L_k})=k\!-\!1$
%(so $k = \ell$ and $b \not\in \range(A)$),
$k = \ell$ and $b \not\in \range(A)$,
the last row and column of $L_k$ are zero, % see \eqref{qlpeqn3a},
and we only need to solve the last two equations of
  $L_{k-1} u_{k-1} = t_{k-1}$, where
  \begin{equation}  \label{eq:Lut}
     L_k = \bmat{L_{k-1} \\ 0 & 0}  \text{,}\quad
     u_k = \bmat{u_{k-1} \\ 0}      \text{,}\quad
     \bmat{\gamma_{k-2}^{(6)}
        \\ \vartheta_{k-1}^{(2)} & \gamma_{k-1}^{(5)}
          }
     \bmat{\mu_{k-2}^{(3)}
        \\ \mu_{k-1}^{(2)}
          }
   = \bmat{{\bar\tau}_{k-2}
        \\ {\bar\tau}_{k-1}
       }.
  \end{equation}
\end{enumerate}
The corresponding solution estimate is
  $x_k = V_k y_k = V_k P_k u_k = W_ku_k$, where
\begin{align}
  W_k \equiv V_k P_k
         &=  \bmat{ V_{k-1} P_{k-1} &  v_k } P_{k-2,k} P_{k-1,k}   \label{eq:wvp}
 \\      &=  \bmat{ W_{k-3}^{(4)} &  w_{k-2}^{(3)} & w_{k-1}^{(2)} & v_k} P_{k-2,k} P_{k-1,k} \nonumber
 \\      &=  \bmat{ W_{k-3}^{(4)} &  w_{k-2}^{(4)} & w_{k-1}^{(3)} & w_k^{(2)}},              \nonumber
 \\  W_k^TW_k &= I_k, \qquad \range(W_k)=\mathcal{K}_k(A,b),                            \label{W_ortho}
\end{align}
and we update
$x_{k-2}$ and compute $x_k$ by short-recurrence orthogonal steps:
\begin{align}
x_{k-2}^{(2)}  &= x_{k-3}^{(2)} + w_{k-2}^{(4)} \mu_{k-2}^{(3)}
\text{, where } x_{k-3}^{(2)} \equiv W_{k-3}^{(4)} u_{k-3}^{(3)}, \label{qlpeqnsol1}
\\     x_k    &= x_{k-2}^{(2)} + w_{k-1}^{(3)} \mu_{k-1}^{(2)}
                      + w_k    ^{(2)} \mu_k.      \label{qlpeqnsol2}
\end{align}

\subsection{Termination}  \label{sec:term}

When $k=\ell$, $Q_{k,k+1}$ is not formed or applied,
see \eqref{QRfac} and \eqref{min7}, and the QR factorization stops.
In \MINRESQLP, we still need to apply $P_{k-2,k}P_{k-1,k}$ on the
right to obtain the minimum-length solution; see Figure~\ref{QLPfig}.

%\begin{proposition} \label{min-len-krylov}
%If $\hat{x}$ is the minimum-length solution of $Ax \approx b$ in a Krylov subspace, then $\hat{x} = x^\dagger$.
%\end{proposition}
%\begin{proof}
%Let $\hat{x} = \arg\min\{\norm{x} \mid A\T Ax=A\T b, \; x \in \mathcal{K}_k(A,b)\}$.  First we want to show that $\hat{x} \in \range(A)$. Suppose $\hat{x}$ has a non-zero component in $\Null(A)$. Then $\hat{x} = x_r + x_n$ where $x_r \in \range(A)$ and $0\ne x_n\in \Null(A)$, and $\norm{\hat{x}}^2 = \norm{x_r}^2 + \norm{x_n}^2 > \norm{x_r}^2$ and $A\T A x_r = A\T b$. Thus $x_r$ is the min-length solution of $\norm{Ax-b}$ in the $k$-th Krylov subspace, which is a contradiction.
%
%We know that $x^\dagger$ is unique and
%$x^\dagger = \arg\min\{\norm{x} \mid A\T Ax=A\T b, \; x \in \mathbb{R}^n\}
%= \arg\min\{\norm{x} \mid A\T Ax=A\T b, \; x \in \range(A)\}$. Since $\hat{x} \in \range(A)$, we must have $\hat{x} = x^\dagger$.
%\end{proof}

\begin{theorem}[{\cite[Theorem 3.1]{C06}}]  \label{theorem-MINRES-QLP}
  In \MINRESQLP{}, $x_\ell = x^\dagger$. %is the minimum-length solution of $Ax \approx b$.
\end{theorem}

\begin{proof}
%This can be seen from \eqref{Lsubproblem}--\eqref{eq:Lut} with $k=\ell$.
  When $b \in \range(A)$, the proof is the same as that for
  Theorem~\ref{theorem-singular-compatible}.

  When $b \notin \range(A)$, for all $u = [u_{\ell-1} \ \,\mu_k]\T
  \in \mathbb{R}^\ell$ that solves \eqref{Lsubproblem}, \MINRESQLP{}
  returns the min-length LS solution $u_\ell = [u_{\ell-1} \ \,0]\T$
  by the construction in \eqref{eq:Lut}.  For any
  $x\in\range(W_{\ell}) = \mathcal{K}_{\ell}(A,b)$ by (\ref{W_ortho}),
\begin{align*}
  \norm{Ax-b} &= \norm{AW_\ell u - b} = \norm{AV_\ell P_\ell u - b}
    = \norm{V_\ell T_\ell P_\ell u - \beta_1V_\ell e_1}
    =       \norm{T_\ell P_\ell u - \beta_1 e_1}
\\ &=\normm{Q_{\ell-1} T_\ell P_\ell u - \bmat{t_{\ell-1} \\ \phi_{\ell-1}}} =
   \normm{\bmat{L_{\ell-1} & 0 \\ 0 & 0} u -\bmat{t_{\ell-1} \\ \phi_{\ell-1}}}.
\end{align*}
Since $L_{\ell-1}$ is nonsingular, $\phi_{\ell-1} = \min \norm{Ax-b}$
can be achieved by $x_\ell = W_\ell u_\ell = W_{\ell-1} u_{\ell-1}$
and $\norm{x_\ell} = \norm{W_{\ell-1} u_{\ell-1}} = \norm{u_{\ell-1}}$
by (\ref{W_ortho}). Thus  $x_\ell$ is the min-length LS solution of $\norm{Ax-b}$ in
$\mathcal{K}_{\ell}(A,b)$, i.e., $x_\ell = \arg\min\{\norm{x} \mid A^2 x=A b, \; x \in \mathcal{K}_{\ell}(A,b)\}$.
Likewise $y_\ell = P_\ell u_\ell$ is the min-length LS solution of $\norm{T_\ell y - \beta_1 e_1}$ and so $y_\ell \in \range(T_\ell)$, i.e. $y_\ell = T_\ell z$ for some $z$. Thus
$x_\ell = V_\ell y_\ell = V_\ell T_\ell z = A V_\ell z \in \range(A)$.
We know that $x^\dagger = \arg\min\{\norm{x} \mid A^2 x=A b, \; x \in \mathbb{R}^n\}$ is unique and $x^\dagger \in \range(A)$.  Since $x_\ell \in \range(A)$, we must have $x_\ell = x^\dagger$.
%, i.e., $x_\ell$ is the min-length LS solution of $\norm{Ax-b}$ in $\mathbb{R}^n$.
\end{proof}

\subsection{Transfer from \MINRES{} to \MINRESQLP{}}
\label{transfer}

On well-conditioned systems, \MINRES{} and \MINRESQLP{} behave very
similarly. However, %\MINRES{} is cheaper in terms of memory and flops.
\MINRESQLP{} requires one more vector of storage, and each iteration
needs 4 more axpy's ($y \leftarrow \alpha x + y$) and 3 more vector
scalings ($x \leftarrow \alpha x$).  Thus it would be a desirable
feature to invoke \MINRESQLP{} from \MINRES{} only if $A$ is
ill-conditioned or singular.  The key idea is to transfer %from \MINRES{}
to \MINRESQLP{} at an iteration where $\underTk$ %has full rank and
is not yet too ill-conditioned.  The \MINRES{} and
\MINRESQLP{} solution estimates are the same, so from
\eqref{minresxk}, \eqref{qlpeqnsol2}, and \eqref{Lsubproblem}: $ x_k^M
= x_k \Longleftrightarrow D_k t_k = W_k u_k = W_k L_k^{-1} t_k$.  Now
from \eqref{eq:rdeqv}, \eqref{qlpeqn3a}, and \eqref{eq:wvp},
\begin{equation}  \label{transfereqn1}
  D_k L_k=(V_kR_k^{-1})(R_kP_k)=V_kP_k=W_k,
\end{equation}
and the last three columns of $W_k$ can be obtained from the last
three columns of $D_k$ and $L_k$.  (Thus, we transfer the three
\MINRES{} basis vectors $d_{k-2}, d_{k-1}, d_k$ to $w_{k-2}, w_{k-1},
w_k$.)  In addition, we need to generate $\smash{x_{k-2}^{(2)}}$
using \eqref{qlpeqnsol1}:
\[
  x_{k-2}^{(2)} = x_k^M - w_{k-1}^{(3)} \mu_{k-1}^{(2)} - w_k^{(2)} \mu_k .
\]

It is clear from \eqref{transfereqn1} that we still need to do the
right transformation $R_k P_k = L_k$ in the \MINRES{} phase and keep
the last $3 \times 3$ principal submatrix of $L_k$ for each $k$ so
that we are ready to transfer to \MINRESQLP{} when necessary.  We then
obtain a short recurrence for $\norm{x_k}$ (see
section~\ref{subsectsolnorm}) and for this computation we save flops
relative to the original \MINRES{} algorithm, where $\norm{x_k}$ is
computed directly.

In the implementation, the \MINRES{} iterates transfer to \MINRESQLP{}
iterates when an estimate of the condition number of $T_k$ (see
\eqref{cond2AQLP}) exceeds an input parameter $\mathit{trancond}$.
Thus, $\mathit{trancond} > 1/\varepsilon$ leads to \MINRES{} iterates
throughout, while $\mathit{trancond} = 1$ generates \MINRESQLP{}
iterates from the start.

\subsection{Comparison of Lanczos-based solvers}  \label{sec:compare}

We compare \MINRESQLP{} with \CG, \SYMMLQ, and \MINRES{} in Tables
\ref{tableQLPsubproblems}--\ref{tableQLPbasis} in terms of subproblem
definitions, basis, solution estimates, flops and memory.  A careful
implementation of \SYMMLQ{} provides a point in
$\mathcal{K}_{k+1}(A,b)$ as shown.  All solvers need storage for
$v_k$, $v_{k+1}$, $x_k$, and a product $p_k = Av_k$ each iteration.
Some additional work-vectors are needed for each method (e.g.,
$d_{k-1}$ and $d_k$ for \MINRES, giving 7 work-vectors in total).

\begin{table}[t!]    %%% Table 5.1
\caption{Subproblems defining $x_k$ for CG, SYMMLQ, MINRES, and MINRES-QLP.}
\label{tableQLPsubproblems}
\centering
\begin{tabular}{|l|l|l|l|}
   \hline
   Method  &  Subproblem  &  Factorization  &  Estimate of $x_k$
\\ \hline \tablestrut
   cgLanczos            & $T_k y_k = \beta_1 e_1$ & Cholesky: & $x_k^C = V_k y_k$
\\ \cite{HS52,PS75,SOL} &                         & $T_k = L_k D_k L_k\T$
                                                  & $\quad \in \mathcal{K}_k(A,b)$
\\[0pt] \hline \tablestrut
   \SYMMLQ  & $y_{k+1} = \arg\min_{y \in \mathbb{R}^{k+1}} \norm{y}$
                        & LQ:                     & $x_k^L = V_{k+1} y_{k+1}$
\\ \cite{PS75,S95}      & \quad s.t.~$\underline{T_k\T} y = \beta_1 e_1$
                        & $\underTk\T  Q_k\T  = \bmat{L_k & \!\!\!0}$
                        & \quad $\in \mathcal{K}_{k+1}(A,b)$
\\[2pt] \hline \tablestrut
  \MINRES               & ${\displaystyle y_k =
                          \arg\min_{y\in\mathbb{R}^k}
                          \norm{\underTk y - \beta_1 e_1}}$
                               & QR:     & $x_k^M = V_k y_k$
\\[-10pt]  \cite{PS75}  &      & $Q_k\underTk = \raisebox{4pt}{$\bmat{R_k \\ 0}$}$
                                         & ${\quad\in\mathcal{K}_k}{(A,b)}$
\\[2pt] \hline \tablestrut
   \MINRESQLP\!\!       & $y_k = \arg\min_{y\in\mathbb{R}^k} \norm{y}$
                               & {QLP:}  & $x_k^Q =V_k y_k$
\\[-4pt] \cite{C06}     & s.t.~$y \in \arg\min \norm{\underTk y - \beta_1 e_1}$\!\!
                               & $Q_k\underTk P_k = \raisebox{4pt}{$\bmat{L_k\\0}$}$
                                         & ${\quad \in \mathcal{K}_k}{(A,b)}$
\\[0pt] \hline
\end{tabular}
\end{table}

\begin{table}[t!]    %%% Table 5.2
\caption{Bases, subproblem solutions, storage, and work for each method.}
\label{tableQLPbasis}
\centering
\begin{tabular}{|l|l|l|l|c|}
   \hline
   Method & New basis & $ \quad\quad\quad z_k, t_k, u_k $ & $x_k$ estimate
                      & \!\!vecs \ flops\!\!
\\ \hline \tablestrut
   cgLanczos & $W_k \equiv V_k L_k^{-T}$ & $ L_k D_k z_k =\beta_1 e_1$
             & $x_k^C \!=\! W_k z_k$
             & 5 \ \ $\ 8n$
\\ \hline  \rule[0ex]{0ex}{4ex}%
   \SYMMLQ   & $W_k\equiv V_{k+1} Q_k^T \bmat{I_k \\ 0}$\!\! & $L_k z_k \!=\!\beta_1 e_1$
             & $x_k^L \!=\! W_k z_k$
             & 6 \ \ $\ 9n$
\\[8pt] \hline \tablestrut
   \MINRES   & $D_k \equiv V_k R_k^{-1}$
             & $\ \ \ \ t_k \!=\!\beta_1 \mystrut\bmat{I_k & \!\!\!0} Q_k e_1$\!\!
             & $x_k^M \!=\! D_k t_k $
             & 7 \ \ $\ 9n$
\\ \hline \tablestrut
   \MINRESQLP\!\! & $W_k \equiv V_k P_k$
             & $L_k u_k \!=\!\beta_1 \mystrut\bmat{I_k & \!\!\!0} Q_k e_1$\!\!
             & $x_k^Q \!=\! W_k u_k$
             & 8 \ \ $14n$
\\ \hline
\end{tabular}
\end{table}

\section{Stopping conditions and norm estimates}  \label{sec:QLPstop}

This section derives several norm estimates that are computed in
\MINRESQLP{}.  As before, we assume exact arithmetic throughout, so
that $V_k$ and $Q_k$ are orthonormal.
Table~\ref{tab-stopping-conditions} summarizes how the norm estimates
are used to formulate three groups of stopping conditions.
The second NRBE test $\norm{Ar_k} \le \norm{A} \norm{r_k} \mathit{tol}$
is from Stewart \cite{Stewart-Rice-1977} with symmetric $A$.

\begin{table}[tb]    %%% Table 6.1
\caption{Stopping conditions in MINRES-QLP.
   NRBE means normwise relative backward error, and
   $\mathit{tol}$, $\mathit{maxit}$, $\mathit{maxcond}$, and
   $\mathit{maxxnorm}$ are input parameters.  All norms and
   $\kappa(A)$ are estimated by
   MINRES-QLP.}
\label{tab-stopping-conditions}
\begin{center}
  \begin{tabular}{|l|l|l|}
    \hline
     Lanczos & NRBE
             & Regularization attempts
  \\ \hline $ $ & &
  \\[-8pt]
     $\beta_{k+1} \le n\norm{A} \varepsilon$
             & $\norm{r_k} / \left(\norm{A}\norm{x_k}+\norm{b}\right)
                   \le \mathit{tol}$
             & $\kappa(A) \geq \mathit{maxcond}$
  \\[4pt]
     $k = \mathit{maxit}$
             & $\norm{Ar_k} / \left(\norm{A} \norm{r_k} \right)
                   \le \mathit{tol}$
             & $\norm{x_k} \geq  \mathit{maxxnorm}$
  \\ \hline
  \end{tabular}
\end{center}
\end{table}

\subsection{Residual and residual norm}

First we derive recurrence relations for $r_k$ and its norm
$\phi_k \equiv \norm{r_k}$.
%In practice, $\rank(L_\ell)$ will be a reliable measure of $\rank(T_\ell)$.

%The following lemma says, among other things, that the intermediate
%$r_k$'s in \MINRESQLP{} are \emph{not} orthogonal to $\mathcal{K}_k(A,b)$.

\begin{lemma}[$r_k$ and $\norm{r_k}$ for \MINRESQLP{}
              and monotonicity of $\norm{r_k}$]
  \label{minresqlp_r}
%\[
%  \begin{array}{@{}l@{\quad}l@{\quad}l@{}}
%     \text{If } k < \ell, \text{ then } \rank(L_k)=k,
%        & r_k    = s_k^2 r_{k-1}\! -\! \phi_k c_k v_{k+1}, & \phi_k = \phi_{k-1} s_k;
%  \\ \text{If } k=\ell\ \&\ \rank(L_\ell)=\ell,       \text{ then}
%        & r_\ell = 0;                                   &
%  \\ \text{If } k=\ell\ \&\ \rank(L_\ell)=\ell\!-\!1, \text{ then}
%        & r_\ell = r_{\ell-1} \ne 0,
%        & \norm{r_\ell}  = \phi_{\ell-1};
%\end{array}
%\]
%where $\ \norm{r_k}=\phi_k$ when $k < \ell$.
%It follows that $\norm{r_k} \le \norm{r_{k-1}}$.
\begin{itemize}
\item If $k < \ell$, then $\rank(L_k) = k$,
         $r_k = s_k^2 r_{k-1} \!-\! \phi_k c_k v_{k+1}$,
         and $ \phi_k = \phi_{k-1} s_k > 0$.
\item If $\rank(L_\ell) = \ell$, then $r_\ell=0$.
\item If $\rank(L_\ell) = \ell\!-\!1$, then
         $r_\ell = r_{\ell-1} \ne 0$, and
         $\norm{r_\ell}  = \phi_{\ell-1} > 0$.
\end{itemize}
\end{lemma}

\begin{proof}
  If $k < \ell$, %$R_k$ is nonsingular and
  the residual is the same as for \MINRES{}.
  We have $\norm{r_k} = \phi_k = \phi_{k-1}s_k > 0$;
  see \eqref{QRfac2}--\eqref{eq:normrk}.
  Also from $r_k = \phi_k V_{k+1} Q_k^T e_{k+1}$ \eqref{rk7} we have
\begin{align}
    r_k  &= \phi_k \bmat{V_k & v_{k+1}}
                   \bmat{Q_{k-1}\T & \\ & 1}
                   \hbox{\footnotesize $
                   \bmat{I_{k-1} & & \\ & c_k & s_k\\ & s_k & -c_k}
                   \bmat{0 \\ 0 \\ 1}
                   $}
            \quad \mbox{by \eqref{QRfac}},    \nonumber
\\[-3pt] \displaybreak[0]
         &= \phi_k \bmat{V_k & v_{k+1}}
                   \bmat{Q_{k-1}\T & \\ & 1}
                   \bmat{s_k e_k\\ -c_k}
           =\phi_k \bmat{V_k & v_{k+1}}
                   \bmat{s_kQ_{k-1}\T e_k\\ -c_k}
   \nonumber %\label{rk7b}
\\[-3pt] \displaybreak[0]
         &= \phi_k s_k V_k Q_{k-1}\T e_k - \phi_k c_k v_{k+1}
          = \phi_{k-1} s_k^2 V_k Q_{k-1}\T e_k - \phi_k c_k v_{k+1}
   \nonumber
\\ \displaybreak[0]
         &= s_k^2 r_{k-1} - \phi_k c_k v_{k+1} \text{ by \eqref{rk7}}.
   \nonumber %\label{rk8}
\end{align}
% If $k = \ell$, from Theorem~\ref{theorem-MINRES-QLP}, %Remark~\ref{rem:ns},
% $b\in \mathcal{R}(A)$ if and only if $T_k$ is nonsingular.
% But $\rank(L_k)=k$ implies $R_k$ and $T_k$ are
% nonsingular, and so $b\in \mathcal{R}(A)$ and $r_k=0$.
%If $\rank(L_\ell)=\ell\!-\!1$,
If $T_\ell$ is nonsingular, $r_\ell = 0$.
Otherwise $Q_{\ell-1,\ell}$ has made the last row of
$R_\ell$ zero, so the last row and column of $L_\ell$ are zero; see
\eqref{eq:Lut}.  Thus $r_\ell = r_{\ell-1}\ne 0$; see Remark~\ref{rem:stop2}.
\end{proof}

\subsection{Norm of $Ar_k$}

Next we derive recurrence relations for $Ar_k$ and its norm
$\psi_k \equiv \norm{Ar_k}$, and we show that $Ar_k$ is orthogonal
to $\mathcal{K}_k(A,b)$.

\begin{lemma}[$Ar_k$ and $\psi_k \equiv \norm{Ar_k}$ for \MINRESQLP{}]
  \label{minresqlp_Ar}

%If $\beta_{k+1}>0$, then
%\[
%  \begin{array}{@{}l@{\quad}l@{\quad}l@{}}
%     \rank(L_k)=k,
%     &   Ar_k = \norm{r_k} (\gamma_{k+1}v_{k+1}
%                          + \delta_{k+2}v_{k+2}),
%     & \psi_k = \norm{r_k} \norm{[\gamma_{k+1} \ \; \delta_{k+2}]},
%  \end{array}
%\]
%where $\delta_{k+2}v_{k+2}$ is zero if $\beta_{k+2}=0$.
%\[
%  \begin{array}{@{}l@{\quad}l@{\quad}l@{}}
%     \text{If } \beta_{k+1}=0\ \&\ \rank(L_k)=k,
%     &   Ar_k = 0,
%     & \psi_k = 0.
%  \\ \text{If } \beta_{k+1}=0\ \&\  \rank(L_k)=k\!-\!1,
%     &   Ar_k = Ar_{k-1},
%     & \psi_k = \psi_{k-1}.
%  \end{array}
%\]

\begin{itemize}
\item If $k < \ell$, then $\rank(L_k) = k$,
         $Ar_k = \norm{r_k} (\gamma_{k+1}v_{k+1} + \delta_{k+2}v_{k+2})$
         and $\psi_k = \norm{r_k} \norm{[\gamma_{k+1} \ \; \delta_{k+2}]}$,
         where $\delta_{k+2}=0$ if $k = \ell\!-\!1$.
\item If $\rank(L_\ell) = \ell$, then $Ar_\ell=0$ and $\psi_\ell = 0$.
\item If $\rank(L_\ell) = \ell\!-\!1$, then
         $Ar_\ell = Ar_{\ell-1} = 0$, and
         $\norm{\psi_\ell}  = \psi_{\ell-1} = 0$.
\end{itemize}
\end{lemma}

\begin{proof}
%These follow from Lemmas~\ref{minreslemma2} and \ref{minresqlp_r}.
For the first case, the proof is essentially the same as the proof of Lemma~\ref{minreslemma2}.
For the other two cases, the results follow directly from Lemma~\ref{minresqlp_r}.
\end{proof}

\subsection{Matrix norms}

For Lanczos-based algorithms,
$\norm{A} \geq \norm{V_{k+1}^TAV_k}=\norm{\underTk}$.
Define
\begin{equation}  \label{normA2}
  \mathcal{A}^{(0)} \equiv 0, \quad \mathcal{A}^{(k)}
    \equiv \max_{j=1,\ldots,k} \left\{ \norm{\underTj e_j} \right\}
    = \max\left\{\mathcal{A}^{(k-1)},\norm{\underTk e_k} \right\}
      \text{ for } k \ge 1.
\end{equation}
Then $\norm{A} \geq \norm{\underTk} \geq \mathcal{A}^{(k)}$.
Clearly, $\mathcal{A}^{(k)}$ is monotonically increasing and is thus
an improving estimate for $\norm{A}$ as $k$ increases.  By the
property of QLP decomposition in \eqref{qlpeqn1a} and \eqref{qlpRightRef},
we could easily extend \eqref{normA2} to include the largest diagonal
of $L_k$:
\begin{equation}  \label{normA2b}
  \mathcal{A}^{(0)} \equiv 0, \quad
  \mathcal{A}^{(k)} \equiv \max\{\mathcal{A}^{(k-1)},       \,
                                   \norm{ \underline {T_k}e_k}, \,
                                   \gamma_{k-2}^{(6)},          \,
                                   \gamma_{k-1}^{(5)},          \,
                                  |\gamma_k^{(4)}| \} \text { for } k\ge1.
\end{equation}
Some other schemes inspired by Larsen \cite[section~A.6.1]{L98},
Higham \cite{H}, and Chen and Demmel \cite{CD00} follow.
For the latter scheme, we use an implementation by Kaustuv \cite{K08}
for estimating the norms of the rows of $A$.
\begin{enumerate}
\item \cite{L98} $\norm{T_k}_1 \ge \norm{T_k} $
\item \cite{L98} $\sqrt{\norm{\underTk^T \underTk}_1} \ge \norm{T_k}  $
\item \cite{L98} $\norm{T_j} \le \norm{T_k}$  for small $j = 5$ or $20$
\item \cite{H}   \Matlab{} function {\small NORMEST}$(A)$, which is based on the
      power method
\item \cite{CD00}
$\max_i \norm{h_i}/\sqrt{m}$, where
$h_i^T$ is the $i$th row of $AZ$,
each column of $Z \in \mathbb{R}^{n \times m}$ is a random vector of $\pm1$'s,
and $m$ is a small integer (e.g., $m=10$).
\end{enumerate}

Figure~\ref{AnormQLP} plots estimates of $\norm{A}$ for 12 matrices
from the Florida sparse matrix collection \cite{UFSMC} whose sizes $n$
vary from 25 to 3002.  In particular, scheme 3 above with $j = 20$
gives significantly more accurate estimates than other schemes for the
12 matrices we tried.  However, the choice of $j$ is not always clear
and the scheme
adds a little
to the cost of \MINRESQLP{}.  Hence we propose incorporating it into
\MINRESQLP{} (or other Lanzcos-based iterative methods)
if very accurate $\norm{A}$ is needed.  Otherwise \eqref{normA2b} uses
quantities readily available from \MINRESQLP{} and gives us
satisfactory estimates for the order of $\norm{A}$.

\begin{figure}[tb]    %%% Figure 6.1
\centering
\hspace*{-1em}
\includegraphics[width=1.00\textwidth]{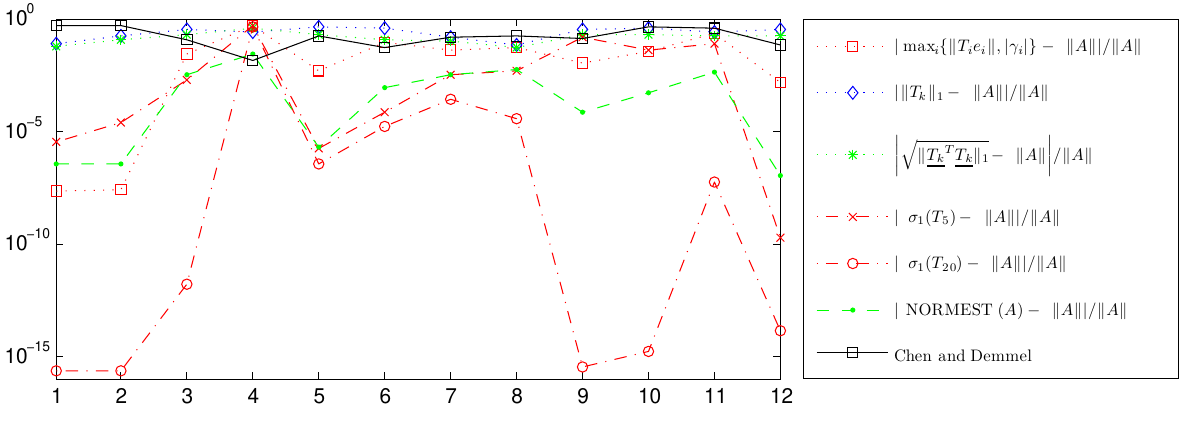}
\caption{Relative errors in different estimates of $\norm{A}$.
  This figure can be reproduced by \texttt{testminresQLPNormA8}.}
\label{AnormQLP}
\end{figure}

\subsection{Matrix condition numbers}

We again apply the property of the QLP decomposition in
\eqref{qlpeqn1a} and \eqref{qlpRightRef} to estimate
$\kappa(\underTk)$, which is a lower bound for
$\kappa(A)$:
\begin{align}
  \gamma_{\min} &\leftarrow \min\{\gamma_1, \gamma_2^{(2)}\}, \quad
  \gamma_{\min}  \leftarrow \min\{\gamma_{\min}, \gamma_{k-2}^{(6)}, \,
                               \gamma_{k-1}^{(5)}, \,  |\gamma_k^{(4)}| \}
                            \text{ for } k \ge 3, \nonumber
\\ \kappa^{(0)} &\equiv 1, \quad
   \kappa^{(k)} \equiv \max\left\{ \kappa^{(k-1)},
                                   \frac{\mathcal{A}^{(k)}}{\gamma_{\min}}
                          \right\}  \text{ for } k \ge 1.
   \label{cond2AQLP}
\end{align}

\subsection{Solution norms} \label{subsectsolnorm}

We derive a recurrence relation for $\norm{ x_k} $ whose cost is as
low as computing the norm of a $3$- or $4$- vector.

Since $\norm{x_k} = \norm{V_k P_k u_k} = \norm{u_k}$, we can estimate
$\norm{x_k}$ by computing $\chi_k \equiv \norm{u_k}$.  However, the
last two elements of $u_k$ change in $u_{k+1}$ (and a new element
$\mu_{k+1}$ is added).  We therefore maintain $\chi_{k-2}$
by updating it
and then using it according to
\[
   \chi_{k-2}^{(2)} =
   \norm{[\chi_{k-3}^{(2)} \ \; \mu_{k-2}^{(3)}]},
   \quad
   \chi_{k}
   = \norm{[\chi_{k-2}^{(2)}  \ \; \mu_{k-1}^{(2)} \ \; \mu_k]}
   \quad \text{cf.~\eqref{qlpeqnsol1} and \eqref{qlpeqnsol2}}.
\]
Thus $\chi_{k-2}^{(2)}$ increases monotonically but we cannot guarantee that
$\norm{x_k} $ and its recurred estimate $\chi_k$ are increasing, and
indeed they are not in some examples (see Figure~\ref{davis1177}).

\subsection{Projection norms}
Sometimes the projection of the right-hand side vector $b$ onto
$\mathcal{K}_k(A,b)$ is required
(for example, see \cite{S97}). A simple recurrence relation is
$\omega_k^2 \equiv \norm{Ax_k}^2 = \omega_{k-1}^2 + \tau_k^2$
and we can derive it in the same way as shown in Lemma~\ref{minreslemma2}.
With $\omega_0 \equiv 0$ we have
 $\omega_k \equiv \norm{Ax_k} =
  \norm{[\omega_{k-1} \ \; \tau_k]}$.

\section{Preconditioned \MINRES{} and \MINRESQLP}  \label{sec:pminres}

It is often asked: How can we construct a preconditioner for a linear
system solver so that the same problem is solved with fewer
iterations?  Previous work on preconditioning the symmetric solvers
\CG{}, \SYMMLQ{}, or \MINRES{} includes \cite{RS02, N90, GMPS92, D98,
FRSW98, NV98, RZ99, MGW00, H01, BL03, TPC04}.%SC: taken out GR01 and added RS02

We have the same question for singular symmetric systems $Ax \approx
b$.  Two-sided preconditioning is needed to
preserve symmetry.  We can still solve compatible systems,
but we will no longer obtain the minimum-length solution.  For
incompatible systems, preconditioning alters the ``least squares''
norm.  To avoid this difficulty we must work with larger equivalent
systems that are compatible.  We consider each case in turn, using a
positive-definite preconditioner $M = CC^T$ with \MINRES{} and
\MINRESQLP{} to solve symmetric compatible systems $Ax = b$.
Implicitly, we are solving equivalent symmetric systems $C\inv AC^{-T}
y = C\inv b$, where $C\T x = y$.  As usual, it is possible to work
with $M$ itself, so without loss of generality we can assume $C =
M^{\myhalf}$.

\subsection{Derivation}

We derive preconditioned \MINRES{} for compatible $Ax = b$ by applying
\MINRES{} to the equivalent problem $\bar{A} \bar{x} = \bar{b}$, where
$\bar{A} \equiv M^{-\myhalf} A M^{-\myhalf}$, $\bar{b} \equiv
M^{-\myhalf} b$, and $x = M^{-\myhalf}\bar{x}$.

\subsubsection{Preconditioned Lanczos process}

Let $v_k$ denote the Lanczos vectors of
$\mathcal{K}(\bar{A},\bar{b})$.  With $v_0 = 0$ and $\beta_1 v_1 =
\bar{b}$, for $k=1,2,\ldots$ we define
\begin{equation}  \label{pminresd4}
   z_k = \beta_k M^{ \myhalf} v_k, \qquad
   q_k = \beta_k M^{-\myhalf} v_k,
   \qquad \text{so that} \quad M q_k =z_k.
\end{equation}
Then
\(
  \beta_k = \norm{\beta_k v_k}
      = \norm{M^{-\myhalf} \! z_k}
      = \norm{ z_k} _{M^{-1}}
      = \norm{ q_k }_M = \sqrt{ q_k\T z_k },
\)
where the square root is well defined because $M$ is positive definite,
and the Lanczos iteration is
\begin{align*}
                  p_k &= \bar{A} v_k = M^{-\myhalf} \! A M^{-\myhalf} v_k
                                     = M^{-\myhalf} \! A q_k / {\beta_k},
\\ \alpha_k           &= v_k^T p_k = q_k\T A q_k / {\beta_k^2},
\\ \beta_{k+1}v_{k+1} &= M^{-\myhalf}\! A M^{-\myhalf}v_k
                            -\alpha_k v_k - \beta_k v_{k-1}.
\end{align*}
Multiplying the last equation by $M^{\myhalf}$ we get
\begin{align*}
   z_{k+1} = \beta_{k+1} M^{\myhalf} v_{k+1}
                & = A M^{-\myhalf} v_k - \alpha_k M^{\myhalf} v_k
                                      - \beta_k  M^{\myhalf} v_{k-1}
\\         &= \frac{1}{\beta_k} A q_k - \frac{\alpha_k}{\beta_k} z_k
                                      - \frac{\beta_k}{\beta_{k-1}} z_{k-1}.
\end{align*}
The last expression involving consecutive $z_j$'s replaces the
three-term recurrence in $v_j$'s.  In addition, we need to
solve a linear system $M q_k = z_k$ \eqref{pminresd4} each
iteration.

\subsubsection{Preconditioned \MINRES}

From \eqref{minresxk} and \eqref{eq:rdeqv} we have the following
recurrence for the $k$th column of $D_k = V_k R_k\inv$ and $\bar{x}_k$:
\[
  d_k = \bigl( v_k - \delta_k^{(2)} d_{k-1}-\epsilon_k d_{k-2} \bigr)
           / \gamma_k^{(2)},
  \qquad
  \bar{x}_k = \bar{x}_{k-1} + \tau_k d_k.
\]
Multiplying the above two equations by $M^{-\myhalf}$ on the left and
defining $\bar{d}_k = M^{-\myhalf}d_k$, we can update the solution of
our original problem by
\[
  \bar{d}_k = \Bigl( \frac{1}{\beta_k} q_k
                - \delta_k^{(2)} \bar{d}_{k-1} - \epsilon_k \bar{d}_{k-2} \Bigr)
                 \!\bigm/\! \gamma_k^{(2)},
  \qquad
  x_k = M^{-\myhalf} \bar{x}_k
      = x_{k-1} + \tau_k \bar{d}_k.
\]
We list the algorithm in \cite[Table~3.4]{C06}.

\subsubsection{Preconditioned \MINRESQLP} \label{secPMINRESQLP}

A preconditioned \MINRESQLP{} can be derived very similarly.  The
additional work is to apply right reflectors $P_k$ to $R_k$, and the
new subproblem bases are $W_k \equiv V_k P_k$, with $\bar{x}_k = W_k
u_k$. Multiplying the new basis and solution estimate by
$M^{-\myhalf}$ on the left, we obtain
\begin{align*}
   \overline{W}_k &\equiv M^{-\myhalf} W_k = M^{-\myhalf} V_k P_k,
\\            x_k & = M^{-\myhalf} \bar{x}_k
                    = M^{-\myhalf} W_k {u}_k
                    = \overline{W}_k {u}_k
                    = x_{k-2}^{(2)} +
                      \mu_{k-1}^{(2)} \bar{w}_{k-1}^{(3)} +
                      \mu_k \bar{w}_k^{(2)}.
\end{align*}
Algorithm~\ref{pminresqlpalgo} lists all steps.  Note that $\bar{w}_k$
is written as $w_k$ for all relevant $k$.  Also, the output $x$ solves
$Ax \approx b$ but the other outputs are associated with $\bar{A}\bar{x}
\approx \bar{b}$.

\paragraph{Remark}
The requirement of positive-definite preconditioners $M$ in \MINRES{}
and \MINRESQLP{} may seem unnatural for a problem with indefinite $A$
because we cannot achieve $M^{-\myhalf}\! A M^{-\myhalf} \approx I$.
However, as shown in \cite{GMPS92}, we can achieve $M^{-\myhalf}\! A
M^{-\myhalf} \approx \left[\begin{smallmatrix}I \\ & -I
  \end{smallmatrix}\right]$ using an approximate
block-LDL$^{\text{T}}$ factorization $A \approx LDL\T$ to get $M =
L|D|L\T$, where $D$ is indefinite with blocks of order 1 and 2, and
$|D|$ has the same eigensystem as $D$ except negative eigenvalues are
changed in sign.

\SQMR{} \cite{FN94} without preconditioning is analytically equivalent
to \MINRES{}.  Unlike \MINRES, \SQMR{} can work directly with an
indefinite preconditioner (such as block-LDL$^{\text{T}}$). However,
in finite precision, \SQMR{} needs ``look-ahead'' to prevent numerical
breakdown.

\begin{algo}{p}

  \Inputs{$A, b, \sigma, M$}

  \smallskip

  $z_0 = 0$,
  \qquad $z_1 = b$,
  \qquad Solve $Mq_1 = z_1$,
  \qquad $\beta_1 = \sqrt{b\T q_1}$
  \tcc*[f]{Initialize}\;

  $w_0 = w_{-1} = 0$,
  \qquad $x_{-2} = x_{-1} = x_0 = 0$\;
  $c_{0,1} \!=\! c_{0,2} \!=\! c_{0,3} \!=\! -1$,
  \quad $s_{0,1} \!=\! s_{0,2} \!=\! s_{0,3}\!=\! 0$,
  \quad $\phi_0 \!=\! \beta_1$,
  \quad $\tau_0 \!=\! \omega_0 \!=\! \chi_{-2} \!=\! \chi_{-1} \!=\! \chi_0 \!=\! 0$\;

  $\delta_1 = \gamma_{-1} = \gamma_0
   = \eta_{-1} = \eta_0 = \eta_1
   = \vartheta_{-1} = \vartheta_0 = \vartheta_1
   = \mu_{-1} = \mu_0 = 0$,
  \quad $\mathcal{A} = 0, \quad  \kappa = 1$\;

  $k=0$\;

  \BlankLine

  \While{no stopping condition is satisfied}{
    $k\leftarrow k+1$\;

    $p_k = Aq_k - \sigma q_k$,
    \qquad $\alpha_k = \frac{1}{\beta_k^2}q_k\T p_k$
    \tcc*[f]{Preconditioned Lanczos}\;

    $z_{k+1} = \frac{1}{\beta_k} p_k - \frac{\alpha_k}{\beta_k}z_k
             - \frac{\beta_k}{\beta_{k-1}} z_{k-1}$\;

    Solve $Mq_{k+1} = z_{k+1}$,
    \qquad $\beta_{k+1} = \sqrt{q_{k+1}\T z_{k+1}}$\;

    \lIf{$k = 1$}{$\rho_k = \norm{[\alpha_k \ \; \beta_{k+1}]}$}
    \lElse{$\rho_k = \norm{[\beta_k \ \; \alpha_k \ \; \beta_{k+1}]}$}\;

    $\delta_k^{(2)} = c_{k-1,1} \delta_k+s_{k-1,1} \alpha_k$
    \tcc*[f]{Previous left reflection\dots}\;

    $\gamma_k = s_{k-1,1}\delta_k -c_{k-1,1} \alpha_k$
    \tcc*[f]{on middle two entries of $\underTk e_k$\dots}\;

    ${\epsilon}_{k+1} = s_{k-1,1} \beta_{k+1}$
    \tcc*[f]{produces first two entries in $\underTkp e_{k+1}$}\;

    $\delta_{k+1} = -c_{k-1,1} \beta_{k+1}$\;

    $c_{k1}, s_{k1}, \gamma_k^{(2)}
      \leftarrow \SymOrtho(\gamma_k, \beta_{k+1})$
    \tcc*[f]{Current left reflection}\;

    $c_{k2}, s_{k2}, \gamma_{k-2}^{(6)}
      \leftarrow \SymOrtho(\gamma_{k-2}^{(5)}, \epsilon_k)$
    \tcc*[f]{First right reflection}\;

    $\delta_k^{(3)} = s_{k2}\vartheta_{k-1} - c_{k2} \delta_k^{(2)}$,
    \qquad $\gamma_k^{(3)} = -c_{k2}\gamma_k^{(2)}$,
    \qquad $\eta_k = s_{k2} \gamma_k^{(2)}$\;

    $\vartheta_{k-1}^{(2)} = c_{k2} \vartheta_{k-1} + s_{k2} \delta_k^{(2)}$\;

    $c_{k3}, s_{k3}, \gamma_{k-1}^{(5)}
      \leftarrow \SymOrtho(\gamma_{k-1}^{(4)}, \delta_k^{(3)})$
    \tcc*[f]{Second right reflection\dots}\;

    $\vartheta_k = s_{k3}\gamma_k^{(3)}$,
    \qquad $\gamma_k^{(4)} = -c_{k3}\gamma_k^{(3)}$
    \tcc*[f]{to zero out $\delta_k^{(3)}$}\;

    $\tau_k = c_{k1} \phi_{k-1}$
    \tcc*[f]{Last element of $t_k$}\;

    $\phi_k = s_{k1} \phi_{k-1}$, \quad $\psi_{k-1} = \phi_{k-1}
         \norm{[\smash{\gamma_k \ \; \delta_{k+1}}]}$
    \tcc*[f]{Update $\norm{r_k}$, $\norm{Ar_{k-1}}$}\;

    \lIf{$k=1$}{$\gamma_{\min}=\gamma_{1}$}
    \lElse{$\gamma_{\min} \leftarrow \min{ \{ \gamma_{\min},
           \gamma_{k-2}^{(6)}, \gamma_{k-1}^{(5)}, | \gamma_k^{(4)}|  \}}$}\;

    $\mathcal{A}^{(k)} = \max{ \{\mathcal{A}^{(k-1)}, \rho_k,
           \gamma_{k-2}^{(6)}, \gamma_{k-1}^{(5)}, |\gamma_k^{(4)}| \}}$
    \tcc*[f]{Update $\norm{A}$}\;

    $\omega_k = \norm{[\smash{\omega_{k-1} \ \; \tau_k}]}$,
    \qquad $\kappa \leftarrow \mathcal{A}^{(k)} / \gamma_{\min}$
    \tcc*[f]{Update $\norm{A x_k}$, $\kappa(A)$}\;

    $w_k = -(c_{k2} / \beta_k) q_k + s_{k2} w_{k-2}^{(3)}$
    \tcc*[f]{Update $w_{k-2}$, $w_{k-1}$, $w_k$}\;

    $w_{k-2}^{(4)} = (s_{k2} / \beta_k) q_k + c_{k2} w_{k-2}^{(3)}$

    \lIf{$k>2$}
        {$w_k^{(2)} = s_{k3} w_{k-1}^{(2)} - c_{k3} w_k$,
         \qquad $w_{k-1}^{(3)} = c_{k3} w_{k-1}^{(2)} + s_{k3} w_k$}\;

    \lIf{$k>2$}
        {$\mu_{k-2}^{(3)} = (\tau_{k-2} - \eta_{k-2} \mu_{k-4}^{(4)}
                                        - \vartheta_{k-2} \mu_{k-3}^{(3)})
                            / \gamma_{k-2}^{(6)}$}
    \tcc*[f]{Update $\mu_{k-2}$}\;

    \lIf{$k>1$}
        {$\mu_{k-1}^{(2)} =(\tau_{k-1} - \eta_{k-1} \mu_{k-3}^{(3)} -
         \vartheta_{k-1}^{(2)} \mu_{k-2}^{(3)}) / \gamma_{k-1}^{(5)}$}
    \tcc*[f]{Update $\mu_{k-1}$}\;

    \lIf{$\gamma_k^{(4)} \neq 0$}
        {$\mu_k = (\tau_k - \eta_k \mu_{k-2}^{(3)}
         - \vartheta_k \mu_{k-1}^{(2)}) / \gamma_k^{(4)}$}
    \lElse{$\mu_k = 0$}
    \tcc*[f]{Compute $\mu_k$}\;

    $x_{k-2}^{(2)}  = x_{k-3}^{(2)}  + \mu_{k-2}^{(3)} w_{k-2}^{(3)}$
    \tcc*[f]{Update $x_{k-2}$}\;

    $x_k = x_{k-2}^{(2)}  + \mu_{k-1}^{(2)} w_{k-1}^{(3)}
                       + \mu_k w_k^{(2)}$
    \tcc*[f]{Compute $x_{k}$}\;

    $\chi_{k-2}^{(2)} = \norm{[\smash{\chi_{k-3}^{(2)}  \ \; \mu_{k-2}^{(3)}}]}$
    \tcc*[f]{Update $\norm{x_{k-2}}$}\;

    $\chi_k = \norm{[\smash{\chi_{k-2}^{(2)}  \ \; \mu_{k-1}^{(2)}
                                       \ \; \mu_k}]}$
    \tcc*[f]{Compute $\norm{x_k}$}\;
  }

  \BlankLine

  $x = x_k$,
  \quad $\phi = \phi_k$,
  \quad $\psi = \phi_k \norm{[\smash{\gamma_{k+1}
                                      \ \; \delta_{k+2}}]}$,
  \quad $\chi = \chi_k$,
  \quad $\mathcal{A} = \mathcal{A}^{(k)}$,
  \quad $\omega = \omega_k$\;

  \Outputs{$x, \phi ,\psi, \chi, \mathcal{A}, \kappa, \omega$}\;

  \tcc*[f]{$c,s \leftarrow \SymOrtho(a,b)$ \rm is a stable form for computing
           $r = \sqrt{a^2+b^2}$,  $c = \frac{a}{r}$,  $s = \frac{b}{r}$}\;

  \caption{Preconditioned MINRES-QLP to solve $(A - \sigma I)x \approx b$.}
  \label{pminresqlpalgo}
\end{algo}

\subsection{Preconditioning singular $Ax = b$}

For singular compatible systems, \MINRES{} and \MINRESQLP{} find the
minimum-length solution (see Theorems
\ref{theorem-singular-compatible} and \ref{theorem-MINRES-QLP}).  If
$M$ is nonsingular, the preconditioned system
is also compatible and the solvers return its minimum-length solution.
The unpreconditioned solution solves $Ax \approx b$, but is not necessarily
a minimum-length solution.

\begin{example}
Let
$ A = \left[\begin{smallmatrix}
     1  &   1  &   0  &   0
  \\ 1  &   1  &   1  &   0
  \\ 0  &   1  &   0  &   1
  \\ 0  &   0  &   1  &   0
            \end{smallmatrix}
      \right]$
and
$ b = \left[\begin{smallmatrix}
              6 \\ 9 \\ 6 \\ 3
            \end{smallmatrix}
      \right].
$
Then $\rank(A)=3$ and $Ax=b$ is a singular compatible
system.  The minimum-length solution is $x^{\dagger} =
\left[\begin{smallmatrix} 2 & 4 & 3 & 2
      \end{smallmatrix}\right]\T$.
By binormalization \cite{LG04}
we construct the matrix
$D = \diag([\begin{smallmatrix} 0.84201 & 0.81228 & 0.30957 & 3.2303
            \end{smallmatrix}] )$.
The minimum-length solution of the diagonally preconditioned problem
 $DAD y \!=\! Db$ is $y^{\dagger} \!=\!
\left[\begin{smallmatrix} 3.5739 & 3.6819 & 9.6909 & 0.93156
  \end{smallmatrix}\right]\T$\!\!.  Then $x  = Dy^{\dagger} =
\left[\begin{smallmatrix} 3.0092 & 2.9908 & 3.0000 & 3.0092
      \end{smallmatrix}\right]\T$ is a solution of $Ax=b$,
but $x \neq x^{\dagger}$.
\end{example}

\subsection{Preconditioning singular $Ax \approx b$}

We propose the following techniques for obtaining minimum-residual
solutions of singular incompatible problems.  In each case we use an
equivalent but larger \emph{compatible} system to which \MINRES{}
may be applied.  Even if the larger system is singular, Theorem
\ref{theorem-singular-compatible} shows that the minimum-length
solution of the larger system will be obtained.  The required $x$
will be part of this solution.  Preconditioning still gives a
minimum-residual solution of $Ax \approx b$, and in \emph{some}
cases $x$ will be $x^\dagger$.
If the systems are ill-conditioned, it will be safer and more
efficient to apply \MINRESQLP{} to the original incompatible system.
However, preconditioning will give an $x$ that is ``minimum length''
in a different norm.

\subsubsection{Augmented system}

When $A$ is singular, so is the augmented system
\begin{align}
   \label{eq:augmented}
   \bmat{I & A \\ A} \bmat{r \\ x} &= \bmat{b \\ 0},
\end{align}
but it is always compatible.  Preconditioning with
symmetric positive-definite $M$ gives us a solution
   $\left[ \begin{smallmatrix} r \\ x \end{smallmatrix} \right]$
in which $r$ is unique, but $x$ may not be $x^\dagger$.

\subsubsection{A giant KKT system}

Problem \eqref{eqn4b}
is equivalent to $\min_{r,\,x} x\T x$ subject to \eqref{eq:augmented},
which is an equality-contrained convex quadratic program.  The
corresponding KKT system \cite[section~16.1]{NW}
is both symmetric and compatible:
\begin{equation}  \label{KKT}
\bmat{  &   & I & A
  \\    &-I & A
  \\  I & A
  \\  A                 }
\bmat{r\\x\\y\\z} = \bmat{0\\0\\b\\0}.
\end{equation}
Although this is still a singular system, the upper-left $3 \times 3$
block-submatrix is nonsingular and therefore
% $\left[ \begin{smallmatrix} r \\ x \\ y \end{smallmatrix} \right]$ is
$r$, $x$, and $y$ are
unique and a preconditioner applied to the KKT system would give
$x$ as the minimum-length solution of our original problem.

\subsubsection{Regularization}

If the rank of a given matrix $A$ is ill-determined,
we may want to solve the \emph{regularized} problem \cite{L77, H90}
with parameter $\delta>0$:
\begin{equation}  \label{regLLS}
  \min_x\ \normm{ \bmat{A \\ \;\delta I\;} x - \bmat{b\\0} }^2.
\end{equation}
The matrix
$\left[\begin{smallmatrix} A \\ \delta I \end{smallmatrix}\right]$
has full rank and is always better conditioned than $A$.
\LSQR{} \cite{PS82a,PS82b} may be applied, and its iterates $x_k$
will reduce $\norm{r_k}^2 + \delta^2\norm{x_k}^2$ monotonically.
Alternatively, we could transform \eqref{regLLS} into the following
symmetric compatible systems and apply \MINRES{} or \MINRESQLP.
They tend to reduce $\norm{Ar_k - \delta^2 x_k}$ monotonically.

\smallskip

\begin{description}

\item[Normal equation:]
  \begin{equation}  \label{form3}
     (A^2 + \delta^2 I) x =Ab.
  \end{equation}

\item[Augmented system:]
\[
   \bmat{I &      A
      \\ A & -\delta^2 I }
   \bmat{r\\x} = \bmat{b\\0}.
\]

\item[A two-layered problem:] If we eliminate $v$ from the system
  \begin{equation}  \label{BV}
    \bmat{ I   & A^2
        \\ A^2 & -\delta^2 A^2}
    \bmat{x\\v} = \bmat{0\\Ab}.
  \end{equation}
  we obtain \eqref{form3}.  Thus $x$ is also a solution of our
  regularized problem \eqref{regLLS}.  This is equivalent to the
  two-layered formulation (4.3) in Bobrovnikova and Vavasis
  \cite{BV01} (with $A_1 = A$, $A_2 = D_1 = D_2 = I$, $b_1 = b$, $b_2
  = 0$, $\delta_1 = 1$, $\delta_2 = \delta^2$).  A key property is
  that $x \rightarrow x^\dagger$ as $\delta \rightarrow 0$.

\item[A KKT-like system:] If we define $y = -Av$ and $r = b - Ax -
  \delta^2 y$, then we can show (by eliminating $r$ and $y$ from the
  following system) that $x$ in
  \begin{equation}  \label{BVreferee7}
   % \hbox{\footnotesize $
     \bmat{  &    &    I       & A
        \\   & -I &    A
        \\ I &  A & \delta^2 I
        \\ A }
     \bmat{r \\x \\y \\v} =\bmat{0 \\ 0 \\b \\0}
   % $}
  \end{equation}
  is also a solution of \eqref{BV} and thus of \eqref{regLLS}. The
  upper-left $3\times 3$ block-submatrix of \eqref{BVreferee7} is
  nonsingular, and the correct limiting behavior occurs: $x
  \rightarrow x^\dagger$ as $\delta \rightarrow 0$.  In fact,
  \eqref{BVreferee7} reduces to \eqref{KKT}.

\end{description}

\subsection{General preconditioners}

The construction of preconditioners is usually problem-dependent.  If
not much is known about the structure of $A$, we can only consider
general methods such as diagonal preconditioning and incomplete
Cholesky factorization.  These methods require access to the nonzero
elements of $A$.  (They are not applicable if $A$ exists only as an
operator for returning the product $Ax$.)

For a comprehensive survey of preconditioning techniques, see Benzi
\cite{B02}.  We discuss a few methods for symmetric $A$ that also
require access to the nonzero $A_{ij}$.

\subsubsection{Diagonal preconditioning}

If $A$ has entries that are very different in magnitude, diagonal
scaling might improve its condition.  When $A$ is diagonally dominant
and nonsingular, we can define $D = \diag(d_1,\ldots,d_n)$ with $d_j =
1/|A_{jj}|^{1/2}$.  Instead of solving $Ax=b$, we solve $DADy=Db$,
where $DAD$ is still diagonally dominant and nonsingular with all
entries $\leq1$ in magnitude, and $x=Dy$.

More generally, if $A$ is not diagonally dominant and possibly
singular, we can safeguard division-by-zero errors by choosing a
parameter $\delta > 0$ and defining
\begin{equation}  \label{DAD}
  d_j(\delta) = 1 / \max \{\delta,
                           \, \sqrt{\smash[b]{|A_{jj}|}},
                           \, \max_{i \ne j} |A_{ij}| \},
    \qquad j=1,\ldots,n.
\end{equation}

\begin{example}
\vspace*{0.05in}
\begin{enumerate}

\item If $A =
  \left[\begin{smallmatrix}
                -1 & 10^{-8} &
        \\ 10^{-8} &   1     & 10^4
        \\         & 10^4    &  0
        \\         &         &      & 0
        \end{smallmatrix}\right]$,
then $\kappa(A)\approx10^{4}$.  Let $\delta = 1$,
  $D =
  \left[\begin{smallmatrix}
           1 &         &         &
        \\   & 10^{-2} &         &
        \\   &         & 10^{-2} &
        \\   &         &         & 1
        \end{smallmatrix}\right]$
in \eqref{DAD}.
Then $DAD =
  \left[\begin{smallmatrix}
           -1 & 10^{-10}      &   &
        \\ 10^{-10} & 10^{-4} & 1 &
        \\          & 1       & 0 &
        \\          &         &   & 0
        \end{smallmatrix}\right]$
and $\kappa(DAD) \approx 1$.

\item $A =
  \left[\begin{smallmatrix}
           10^{-4} & 10^{-8} &         &
        \\ 10^{-8} & 10^{-4} & 10^{-8} &
        \\         & 10^{-8} &    0    &
        \\         &         &         &  0
        \end{smallmatrix}\right]$
contains mostly very small entries, and
$\kappa(A)\approx10^{10}$. Let $\delta = 10^{-8}$ and
$D =
  \left[\begin{smallmatrix}
           10^2 &      &      &
        \\      & 10^2 &      &
        \\      &      & 10^8 &
        \\      &      &      & 10^8
\end{smallmatrix}\right]$.
Then $DAD =
  \left[\begin{smallmatrix}
              1    & 10^{-4} &      &
        \\ 10^{-4} &    1    & 10^2 &
        \\         & 10^2    &  0   &
        \\         &         &      & 0
\end{smallmatrix}\right]$
and $\kappa(DAD) \approx 10^2$.
(The choice of $\delta$ makes a critical difference in this case:
with $\delta=1$, we have $D=I$.)
\end{enumerate}
\end{example}

\subsubsection{Binormalization (BIN)} \label{sectbin}

Livne and Golub \cite{LG04} scale a symmetric matrix by a series of
$k$ diagonal matrices on both sides until all rows and columns of the
scaled matrix have unit $2$-norm:
%\begin{equation}  \label{DADLG04}
$ DAD = D_k\cdots D_{1}AD_{1}\cdots D_k$.
%\end{equation}
See also Bradley \cite{B10}.

\begin{example}
If $A =
  \left[\begin{smallmatrix}
           10^{-8} &    1    &
        \\    1   & 10^{-8}  & 10^4
        \\        & 10^4    & 0
  \end{smallmatrix}\right]$,
then $\kappa(A) \approx 10^{12}$. With just one sweep of BIN, we obtain
$D =\diag(8.1\e{-3},6.6\e{-5},1.5)$,
$DAD \approx
  \left[\begin{smallmatrix}
           6.5 \e{-1} & 5.3 \e{-1}  &  0
        \\ 5.3 \e{-1} &     0       &  1
        \\         0  &     1       &  0
\end{smallmatrix}\right]$
and $\kappa(DAD)\approx 2.6$ even though the rows and columns have not
converged to one in the two-norm. In contrast, diagonal scaling
\eqref{DAD} defined by $\delta=1$ and $D = \diag(1, 10^{-4},10^{-4})$
reduces the condition number to approximately $10^4$.
\end{example}

\subsubsection{Incomplete Cholesky factorization}

For a sparse symmetric positive definite matrix $A$, we could compute
a preconditioner by the incomplete Cholesky factorization that
preserves the sparsity pattern of $A$. This is known as IC0 in
the literature.
Often there exists a permutation $P$ such that the IC0 factor of
$PAP\T$ is more sparse than that of $A$.

When $A$ is semidefinite or indefinite, IC0 may not exist, but a
simple variant that may work is the incomplete Cholesky-infinity
factorization \cite[section~5]{Z97}.

\section{Numerical experiments}  \label{sec:numerical}

We compare the computed results of \MINRESQLP{} and various other
Krylov subspace methods to solutions computed directly by the
eigenvalue decomposition (EVD) and the truncated eigenvalue
decompositions (TEVD) of $A$.  For TEVD we have
\begin{align*}
   x_t \equiv \sum_{|\lambda_i| > t \norm{A} \varepsilon}
               \frac{1}{\lambda_i} u_i u_i\T b,
   \qquad
   \norm{A}  = \max |\lambda_i|,
   \qquad
   \kappa_t(A) = \frac{\max |\lambda_i|}
               {\underset{|\lambda_i| > t \norm{A}\varepsilon} \min |\lambda_i|}
\end{align*}
with parameter $t>0$.  Often $t$ is set to $1$, and sometimes to a
moderate number such as $10$ or $100$; it defines a cut-off point
relative to the largest eigenvalue of $A$.  For example,
if most eigenvalues are of order 1
in magnitude
and the rest are of order $\norm{A}\varepsilon\approx10^{-16}$, we
expect TEVD to work better when the small eigenvalues are excluded,
while EVD (with $t=0$) could return an exploding solution.

In the tables of results, \Matlab{} \MINRES{} and \Matlab{} \SYMMLQ{}
are \Matlab's implementation of \MINRES{} and \SYMMLQ{} respectively.
They incorporate \emph{local reorthogonalization} of the Lanczos
vector $v_2$, which could enhance the accuracy of the computations if
$b$ is close to an eigenvector of $A$ \cite{L76}:
\begin{align*}
   \text{Second Lanczos iteration}  &  \text{: } \beta_1 v_1=b,\text{ and }q_2
      \equiv \beta_2 v_2 = Av_1 - \alpha_1 v_1
\\ \text{Local reorthogonalization} &  \text{: } q_2 \leftarrow
      q_2 -(v_1^T q_2)v_1.
\end{align*}
Lacking the correct stopping condition for singular problems,
\Matlab{} \SYMMLQ{} works more than necessary and then selects the
smallest residual norm from all computed iterates; it would sometimes
report that the method did not converge although the selected estimate
appeared to be reasonably accurate.

\MINRES{} \SOL{} and \SYMMLQ{} \SOL{} are implementations based on
\cite{PS75}.  \MINRES$^+$ and \SYMMLQ$^+$ are the same but with
additional stopping conditions for singular incompatible systems (see
Lemma \ref{minreslemma2} and \cite[Proposition 2.12]{C06}).

The computations in this section were performed on a Windows XP
machine with a 3.2GHz Intel Pentium D Processor 940 and 3GB RAM
($\varepsilon \approx 10^{-16}$) .
Tests were performed with each solver on five types of problem:
\begin{enumerate}
\item symmetric, nonsingular linear systems
\item symmetric, singular linear systems
\item mildly incompatible symmetric (and singular) systems (meaning
  $\norm{r}$ is rather small with respect to $\norm{b}$)
\item symmetric (and singular) LS problems
\item Hermitian systems.
\end{enumerate}

We present a few examples that illustrate the key features of
\MINRESQLP{}.  For a larger set of tests and results, such as applying
\MINRESQLP{} and other Krylov methods to Hermitian systems with
preconditioning, we refer to \cite[Chapter 4]{C06}.

For a compatible system, we generate a random vector $b$ that is in
the range of the test matrix ($b\equiv Ay$, $y_{i} \sim i.i.d.\ U(0,1)$,
i.e., $y_1,\ldots,y_n$ are independent and identically distributed random
variables, whose values are drawn from the standard uniform distribution
with support $[0,1]$). For an LS problem, we
generate a random $b$ that is \emph{not} in the range of the test matrix
($b_{i} \sim i.i.d.\ U(0,1)$ often suffices).

If $A$ is Hermitian, then $v^{H}\!Av$ is real for all complex vectors
$v$. Numerically (in double precision), $\alpha_k=$ $v_k^{H}Av_k$ is
likely to have small imaginary parts in the first few Lanczos
iterations and snowball to have large imaginary parts in later
iterations. This would result in a poor estimation of $\norm{T_k}_{F}$
or $\norm{A}_{F}$, and unnecessary errors in the Lanczos vectors.
Thus we made sure to \emph{typecast} $\alpha_k=$
$\operatorname{real}(v_k^{H}Av_k)$ in \MINRESQLP{} and \MINRES{}
\SOL{}.

We could say from the results that the Lanczos-based methods have
built-in regularization features \cite{KS99}, often matching the TEVD
solutions very well.

\subsection{A Laplacian system $Ax \approx b$ (almost compatible)}

Our first example involves a singular indefinite Laplacian matrix $A$
of order $n=400$.  It is block-tridiagonal with each block being a
tridiagonal matrix $T$ of order $N=20$ with all nonzeros equal to 1:
\begin{equation}  \label{eq:Laplace}
  A = \bmat{T & T
         \\ T &    T   & \ddots
         \\   & \ddots & \ddots & T
         \\   &        &   T    & T}_{n\times n},
  \qquad
  T = \bmat{1 & 1
         \\ 1 & 1      & \ddots
         \\   & \ddots & \ddots & 1
         \\   &        &   1    & 1}_{N\times N}.
\end{equation}
\Matlab's \texttt{eig($A$)} reports the following data:
$205$ positive eigenvalues in the interval $[6.1\e{-2},8.87]$,
$39$ almost-zero eigenvalues in $[-2.18\e{-15},3.71 \e{-15}]$,
$156$ negative eigenvalues in $[-2.91,-6.65 \e{-2}]$,
numerical rank $= 361$.

We used a right-hand side with a small incompatible component: $b = Ay
+ 10^{-8}z$ with $y_i$ and $z_i \sim i.i.d.\ U(0,1)$.  Results are summarized
in Table~\ref{tab:Laplacian1}. In the column labeled ``C?", the value
``Y" denotes that the associated algorithm in the row has converged to
the desired NRBE tolerances within $\mathit{maxit}$ iterations
(cf.~Table~\ref{tab-stopping-conditions}); otherwise, we have values
``N" and ``N?", where ``N?" indicates that the algorithm could have
converged if more relaxed stopping conditions were used. The column
``$Av$" shows the total number of matrix-vector products, and column
``$x(1)$" lists the first element of the final solution estimate $x$
for each algorithm. For \GMRES{}, the integer in parentheses is the value of
the restart parameter.
%SC: omited "defines the number of inner iterations within each outer iteration"

\begin{table}    %%% Table 8.1
\caption{Finite element problem $Ax \approx b$ with $b$ almost compatible.
Laplacian on a $20 \times 20$ grid,
$n = 400$,   $\mathit{maxit} =  1200$,
$\operatorname{shift} = 0$,  $\mathit{tol} = 1.0 \e{-15}$,
$\operatorname{maxnorm} = 100$,  $\mathit{maxcond} = 1 \e{15}$,
$\norm{b} = 87$.
To reproduce this example, run \texttt{test\_minresqlp\_eg7\_1(24)}.}
\label{tab:Laplacian1}
\renewcommand{\e}[1]{\text{e}{#1}}
{\footnotesize \renewcommand{\arraystretch}{1.2}
\begin{tabular}{|l@{}|l@{\,}|r@{\,}|r@{\,}|r@{\,}|r@{\,}|r@{\,}|r@{\,}|r@{\,}|r@{\,}|}
\hline
Method         & C?     & $Av $ & $ x(1)~$ & $\norm{x}~$ & $\norm{e}~$ & $\norm{r}~$ & $\norm{Ar}~$ & $\norm{A}~$ & $\kappa(A)$
\\ \hline
EVD            & --            &  -- & $ -7.39\e{5}$ & $ 4.12\e{7}$ & $ 4.1\e{7}  $ & $1.7\e{-7} $ & $7.8\e{-7}  $ & $8.9\e{0}$ & $ 1.1\e{17}$\\
TEVD           & --      & -- & $  3.89\e{-1}$ & $ 1.15\e{1}$ & $ 0.0\e{0}   $ & $1.7\e{-8} $ & $1.4\e{-12} $ & $8.9\e{0} $ & $ 1.5\e{2}$\\
\Matlab{} SYMMLQ & N?    & $371$ & $  3.89\e{-1}$ & $ 1.15\e{1}$ & $ 1.4\e{-7}  $ & $1.8\e{-7} $ & $5.8\e{-7} $ &  --    &  -- \\
SYMMLQ SOL      & N   &  $447$ & $ -3.08\e{0}$ & $ 9.63\e{1}$ & $ 9.5\e{1}   $ & $1.4\e{2} $ & $4.4\e{2} $ & $9.6\e{1}$ & $ 1.3\e{1}$\\
SYMMLQ$^+$     & N   &  $447$ & $  2.94\e{6}$ & $ 4.27\e{8}$ & $ 4.3\e{8}   $ & $1.8\e{2} $ & $6.5\e{2} $ & $8.6\e{0} $ & $ 1.3\e{1}$\\
\Matlab{}  MINRES  & N &  $1200$ & $ -7.50\e{5}$ & $ 2.10\e{7}$ & $ 2.1\e{7}  $ & $1.5\e{7}$ & $9.1\e{7}$ &  --    &  -- \\
MINRES SOL      & N  &  $1200$ & $  9.89\e{5}$ & $ 6.10\e{7}$ & $ 6.1\e{7}   $ & $2.3\e{7} $ & $1.5\e{8} $ & $1.8\e{2}$ & $ 1.5\e{1}$\\
MINRES$^+$      & N  &  $ 611$ & $  1.02\e{0}$ & $ 9.28\e{1}$ & $ 9.2\e{1}   $ & $1.7\e{-8} $ & $2.5\e{-11} $ & $8.6\e{0}$ & $ 6.9\e{13}$\\
MINRES-QLP      & Y   & $ 612$ & $  3.89\e{-1}$ & $ 1.15\e{1}$ & $ 3.7\e{-11}   $ & $1.7\e{-8} $ & $9.3\e{-11} $ & $8.7\e{0}$ & $ 4.3\e{13}$\\
\Matlab{}  LSQR   & Y   &  $1462$ & $  3.89\e{-1}$ & $ 1.15\e{1}$ & $ 2.3\e{-13}   $ & $1.7\e{-8} $ & $3.3\e{-13} $ &  --    &  --\\
LSQR SOL        & Y  &   $1464$ & $  3.89\e{-1}$ & $ 1.15\e{1}$ & $ 2.4\e{-13}   $ & $1.7\e{-8} $ & $3.9\e{-13} $ & $1.5\e{2}$ & $ 6.4\e{3}$\\
\Matlab{} GMRES(30) & N? & $1200$ & $  3.90\e{-1}$ & $ 1.15\e{1}$ & $ 5.2\e{-2}   $ & $3.4\e{-3} $ & $9.4\e{-4} $ &  --     &   --\\
SQMR              & N & $1200$ & $ -2.58\e{8}$ & $ 3.74\e{10}$ & $ 3.7\e{10}$ & $4.6\e{3} $ & $2.3\e{4} $ &  --     &   --\\
\Matlab{}  QMR     & N?  &  $798$ & $  3.89\e{-1}$ & $ 1.15\e{1}$ & $ 5.2\e{-7}   $ & $1.9\e{-8} $ & $2.6\e{-8} $ &  --     &   --\\
\Matlab{} BICG    & N?  &  $790$ & $  3.89\e{-1}$ & $ 1.15\e{1}$ & $ 4.7\e{-7}   $ & $3.9\e{-8} $ & $1.9\e{-7} $ &  --     &   --\\
\Matlab{}  BICGSTAB & N?  & $2035$ & $  3.89\e{-1}$ & $ 1.15\e{1}$ & $ 4.2\e{-7}   $ & $1.7\e{-8} $ & $4.3\e{-13} $ &  --     &   --\\
\hline
\end{tabular}}
\end{table}

\MINRES{} \SOL{} gives a larger solution than \MINRESQLP{}. This example has
a residual norm of about $1.7\times10^{-8}$, so it is not clear
whether to classify it as a linear system or an LS problem.
To the credit of \Matlab{} \SYMMLQ{}, it thinks the system is linear
and returns a good solution.  For \MINRESQLP, the first 410 iterations
are in standard ``\MINRES{} mode'', with a transfer to ``\MINRESQLP{}
mode'' for the last 202 iterations.  \LSQR{} \cite{PS82a,PS82b}
converges to the minimum-length solution but with more than twice the
number of iterations of \MINRESQLP{}.  The other solvers fall short in
some way.

\subsection{A Laplacian LS problem $\min \norm{Ax-b}$}

This example uses the same Laplacian matrix $A$ \eqref{eq:Laplace} but
with a clearly incompatible $b=10\times\operatorname{rand}(n,1)$,
i.e., $b_{i} \sim i.i.d.\ U(0,10)$. The residual norm is about $17$.  Results
are summarized in Table~\ref{tab:Laplacian2}.  \MINRES{} gives an
LS solution, while \MINRESQLP{} is the only solver that
matches the solution of TEVD.  The other solvers do not perform
satisfactorily.

\begin{table}    %%% Table 8.2
\caption{Finite element problem $\min \norm{Ax-b}$.
Laplacian on a $20 \times 20$ grid,
$n = 400$,   $\mathit{maxit} =  500$,
$\operatorname{shift} = 0$,  $\mathit{tol} = 1.0 \e{-14}$,
$\operatorname{maxnorm} = 1 \e{4}$,  $\mathit{maxcond} = 1 \e{14}$,
$\norm{b} = 120$.
To reproduce this example, run \texttt{test\_minresqlp\_eg7\_1(25)}.}
\label{tab:Laplacian2}
\renewcommand{\e}[1]{\text{e}{#1}}
{\footnotesize \renewcommand{\arraystretch}{1.2}
\begin{tabular}{|l@{}|l@{\,}|r@{\,}|r@{\,}|r@{\,}|r@{\,}|r@{\,}|r@{\,}|r@{\,}|r@{\,}|}
\hline
Method         & C?  & $Av $  & $x(1)~$        & $\norm{x}~$  & $\norm{e}~$ & $\norm{r}~$ & $\norm{Ar}~$ & $\norm{A}~$ & $\kappa(A)$
\\ \hline
EVD             & -- & --     & $-7.39\e{14}$ & $ 4.12\e{16}$& $4.1\e{16}$ & $1.8\e{2} $ & $7.9\e{2}  $ & $8.9\e{0}$ & $ 1.1\e{17}$\\
TEVD            & -- & --     & $-8.75\e{0}$  & $ 1.43\e{2}$ & $0.0\e{0}$  & $1.7\e{1} $ & $4.1\e{-12}$ & $8.9\e{0} $ & $ 1.5\e{2}$\\
\Matlab{} SYMMLQ   & N  &   $1$  & $ 2.74 \e{-1}$& $ 1.52\e{1}$ & $1.4\e{2} $ & $6.0\e{1} $ & $2.9\e{2}$  &  --    &  -- \\
SYMMLQ SOL      & N  & $228$  & $-7.70\e{2}$  & $ 9.93\e{3}$ & $9.9\e{3} $ & $7.0\e{3} $ & $3.4\e{4}$  & $6.8\e{1}$ & $ 9.7\e{0}$\\
SYMMLQ$^+$      & N  & $228$  & $-7.70\e{2}$  & $ 9.93\e{3}$ & $9.9\e{3} $ & $7.0\e{3} $ & $3.4\e{4}$  & $7.6\e{0} $ & $ 9.7\e{0}$\\
\Matlab{} MINRES   & N  & $500$  & $ 2.80\e{14}$ & $ 4.07\e{16}$& $4.1\e{16}$ & $2.3\e{2} $ & $1.4\e{3}$  &  --    &  -- \\
MINRES SOL      & N  & $500$  & $-1.46\e{14}$ & $ 2.11\e{16}$& $2.1\e{16}$ & $1.1\e{2} $ & $6.6\e{2}$  & $1.5\e{2}$ & $ 1.4\e{1}$\\
MINRES$^+$      & N  & $ 381$ & $3.88\e{1}$   & $ 6.90\e{3}$ & $6.9\e{3} $ & $1.7\e{1} $ & $1.2\e{-5}$ & $7.9\e{0}$ & $ 1.6\e{10}$\\
MINRES-QLP      & Y  & $ 382$ & $-8.75\e{0}$  & $ 1.43\e{2}$ & $1.7\e{-6}$ & $1.7\e{1} $ & $1.7\e{-5}$ & $8.6\e{0}$ & $ 3.5\e{10}$\\
\Matlab{} LSQR     & Y  & $1000$ & $-8.75\e{0}$  & $ 1.43\e{2}$ & $2.0\e{-5}$ & $1.7\e{1} $ & $1.4\e{-5}$ &  --    &  --\\
LSQR SOL        & Y  & $1000$ & $-8.75\e{0}$  & $ 1.43\e{2}$ & $2.3\e{-5}$ & $1.7\e{1} $ & $1.1\e{-5}$ & $1.2\e{2}$ & $ 4.4\e{3}$\\
\Matlab{} GMRES(30)& N  &  $500$ & $-8.84\e{0}$  & $ 1.25\e{2}$ & $4.8\e{1} $ & $1.7\e{1} $ & $8.2\e{-1}$ &  --     &   --\\
SQMR            & N  &  $500$ & $ 9.58\e{15}$ & $ 1.39\e{18}$& $1.4\e{18}$ & $1.2\e{11}$ & $6.7\e{11}$ &  --     &   --\\
\Matlab{} QMR      & N  &  $556$ & $-7.30\e{0}$  & $ 1.92\e{2}$ & $1.4\e{2} $ & $1.7\e{1} $ & $1.2\e{1} $ &  --     &   --\\
\Matlab{} BICG     & N  &    $2$ & $ 1.40\e{0}$  & $ 1.71\e{1}$ & $1.4\e{2} $ & $6.0\e{1} $ & $2.6\e{2} $ &  --     &   --\\
\Matlab{} BICGSTAB & N  &  $104$ & $-1.12\e{1}$  & $ 1.40\e{2}$ & $9.6\e{1} $ & $2.6\e{1} $ & $1.8\e{1} $ &  --     &   --\\
\hline
\end{tabular}}
\end{table}

\subsection{Regularizing effect of \MINRESQLP}
\label{sec:regularizing}

This example illustrates the regularizing effect of \MINRESQLP{} with
the stopping condition $\chi_k \le \mathit{maxxnorm}$.  For $k \ge 18$
in Figure~\ref{davis1177}, we observe the following values:
\begin{align*}
\displaybreak[0]
\chi_{18} &= \norm{\bmat{2.51 & \phantom{-} 3.87\e{-11} & \phantom{-}1.38\times 10^2}}
  = 1.38\times 10^2,
\\
\displaybreak[0]
\chi_{19} &= \norm{\bmat{2.51 & -8.00\e{-10} &  -1.52 \times 10^2}}
   = 1.52\times 10^2,
\\
\displaybreak[0]
\chi_{20} &= \norm{\bmat{2.51 & \phantom- 1.62 \e{-10} & -1.62 \times 10^6}}
   = 1.62\times 10^6 > \mathit{maxxnorm} \equiv 10^4.
\end{align*}
Because the last value exceeds $\mathit{maxxnorm}$, \MINRESQLP{}
regards the last diagonal element of $L_k$ as a singular value to be
ignored (in the spirit of truncated SVD solutions).  It discards the
last element of $u_{20}$ and updates
\[
  \chi_{20} \leftarrow \norm{\bmat{2.51 & 1.62 \e{-10} & 0}} = 2.51.
\]

The full truncation strategy used in the implementation is justified
by the fact that $x_k = W_k u_k$ with $W_k$ orthogonal.  When
$\norm{x_k}$ becomes large, the last element of $u_k$ is treated as
zero.  If $\norm{x_k}$ is still large, the second-to-last element of
$u_k$ is treated as zero.  If $\norm{x_k}$ is \emph{still} large, the
third-to-last element of $u_k$ is treated as zero.

\begin{figure}    %%% Figure 8.1
\includegraphics[width=\textwidth]{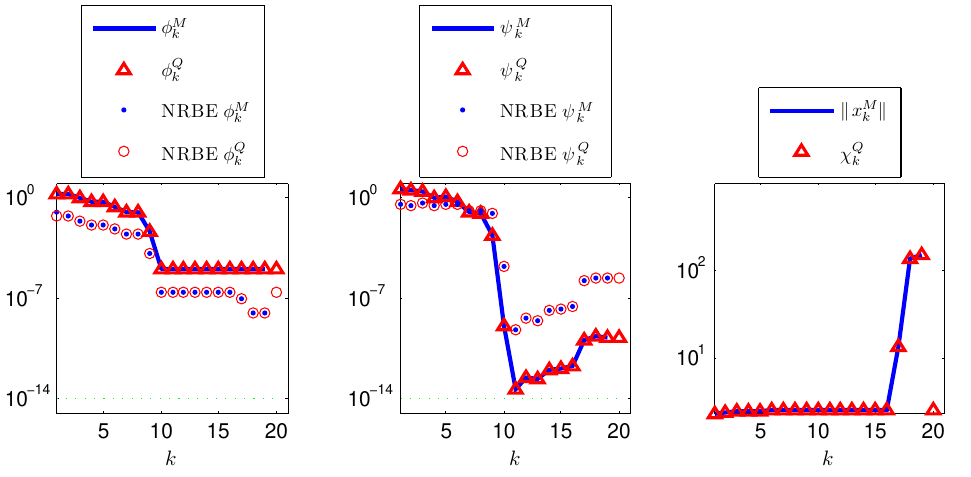}
\caption{Recurred $\phi_k \approx \norm{r_k}$, $\psi_k \approx \norm{Ar_k}$,
  and $\norm{x_k}$ for MINRES and MINRES-QLP.  The matrix $A$ (ID~$1177$
  from \cite{UFSMC}) is positive semidefinite, $n=25$, and $b$ is
  random with $\norm{b} \simeq 1.7$.  Both solvers could have achieved
  essentially the TEVD solution of $Ax\simeq b$ at iteration $11$.
  However, the stringent $\mathit{tol}=10^{-14}$ on the recurred
  normwise relative backward errors (NRBE in
  Table~\ref{tab-stopping-conditions}) prevents them from stopping
  ``in time".  MINRES ends with an exploding solution, yet MINRES-QLP
  brings it back to the TEVD solution at iteration $20$.
  \textbf{Left:} $\phi_k^M$ and $\phi_k^Q$ (recurred $\norm{r_k}$
  of MINRES and MINRES-QLP) and their NRBE.
  \textbf{Middle:} $\psi_k^M$ and $\psi_k^Q$ (recurred
  $\norm{Ar_k}$) and their NRBE.
  \textbf{Right:} $\norm{x_k^M}$ (norms of solution estimates from
  MINRES) and $\chi_k^Q$ (recurred $\norm{x_k}$ from
  MINRES-QLP) with $\mathit{maxxnorm} = 10^{4}$.  This figure can be
  reproduced by \texttt{test\_minresqlp\_fig7\_1(2)}.}
\label{davis1177}
\end{figure}

\subsection{Effects of rounding errors in \MINRESQLP{}}     \label{egDutch}

The recurred residual norms $\phi_k^M$ in \MINRES{} usually
approximate the directly computed ones $\norm{ r_k^M }$ very well
until $\norm{ r_k^M }$ becomes small.  We observe that $\phi_k^M$
continues to decrease in the last few iterations, even though $\norm{
  r_k^M }$ has become stagnant.  This is desirable in the sense that
the stopping rule will cause termination, although the final solution
is not as accurate as predicted.

We present similar plots of \MINRESQLP{} in the following examples,
with the corresponding quantities as $\phi_k^Q$ and $\norm{r_k^Q}$.
We observe that except in very ill-conditioned LS problems,
$\phi_k^Q$ approximates $\norm{r_k^Q}$ very closely.

Figure~\ref{figDPtestSing3_DP} illustrates four singular compatible linear systems.

Figure~\ref{figDPtestLSSing1} illustrates four singular LS problems.

\begin{figure}    %%% Figure 8.2
\centering
\hspace*{-0.1in}
\includegraphics[width=\textwidth]{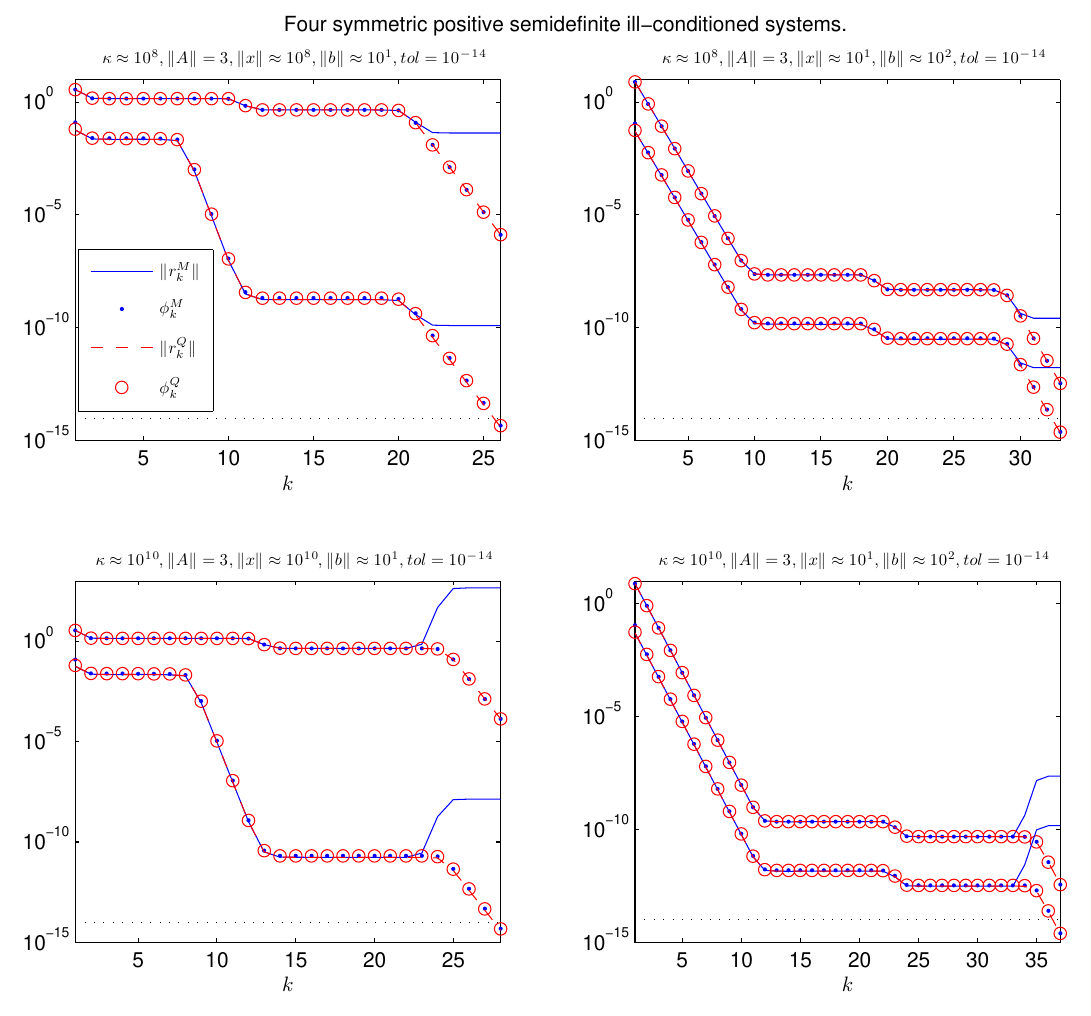}
\vspace*{-0.22in}
\caption[]{Solving $Ax=b$ with semidefinite
  $A$ similar to an example of Sleijpen \etal\ \cite{SVM00}.
  $A = Q \diag([0_5, \eta, 2\eta, 2\!:\! \frac{1}{789}\!:\!3])Q$
  of dimension $n = 797$, nullity 5, and norm $\norm{A}=3$,
  where $Q = I - (2/n) ww\T$ is a Householder matrix generated by
  $v = [0_5, 1, \ldots, 1]\T$, $w = v/\norm{v}$.  These plots
  illustrate and compare the effect of rounding errors in MINRES
  and MINRES-QLP.

  The upper part of each plot shows the computed and recurred residual
  norms, and the lower part shows the computed and recurred normwise
  relative backward errors (NRBE, defined in
  Table~\ref{tab-stopping-conditions}).  MINRES and MINRES-QLP
  terminate when the recurred NRBE is less than the given
  $\mathit{tol} = 10^{-14}$.

\smallskip

\textbf{Upper left:} $\eta=10^{-8}$ and thus $\kappa(A) \approx 10^8$.
Also $b = e$ and therefore $\norm{x} \gg \norm{b}$.  The graphs of
MINRES's directly computed residual norms $\norm{r_k^M}$ and
recurrently computed residual norms $\phi_k^M$ start to differ at the
level of $10^{-1}$ starting at iteration 21, while the values
$\phi_k^Q \approx \norm{r_k^Q}$ from MINRES-QLP decrease monotonically
and stop near $10^{-6}$ at iteration 26.

%\smallskip

\textbf{Upper right:} Again $\eta=10^{-8}$ but $b=Ae$. Thus $\norm{x}=
\norm{e} = O(\norm{b})$.  The MINRES graphs of $\norm{r_k^M}$ and
$\phi_k^M$ start to differ when they reach a much smaller level of
$10^{-10}$ at iteration 30.  The MINRES-QLP $\phi_k^Q$'s are excellent
approximations of $\phi_k^Q$, with both reaching $10^{-13}$ at
iteration 33.

%\smallskip

\textbf{Lower left:} $\eta=10^{-10}$ and thus $A$ is even more
ill-conditioned than the matrix in the upper plots. Here $b= e$ and
$\norm{x}$ is again exploding. MINRES ends with $\norm{r_k^M} \approx
10^2$, which means no convergence, while MINRES-QLP reaches a
residual norm of $10^{-4}$.

%\smallskip

\textbf{Lower right:} $\eta=10^{-10}$ and $b=Ae$. The final MINRES
residual norm $\norm{r_k^M} \approx 10^{-8}$, which is satisfactory
but not as accurate as $\phi_k^M$ claims at $10^{-13}$. MINRES-QLP
again has $\phi_k^Q \approx \norm{r_k^Q} \approx 10^{-13}$ at
iteration 37.

%\smallskip

This figure can be reproduced by \texttt{DPtestSing7.m}.}
\label{figDPtestSing3_DP}

\end{figure}

\begin{figure}[tb]    %%% Figure 8.3
  \centering
%  \vspace*{-0.75in}
  \hspace*{-0.1in}
  \includegraphics[width=\textwidth]{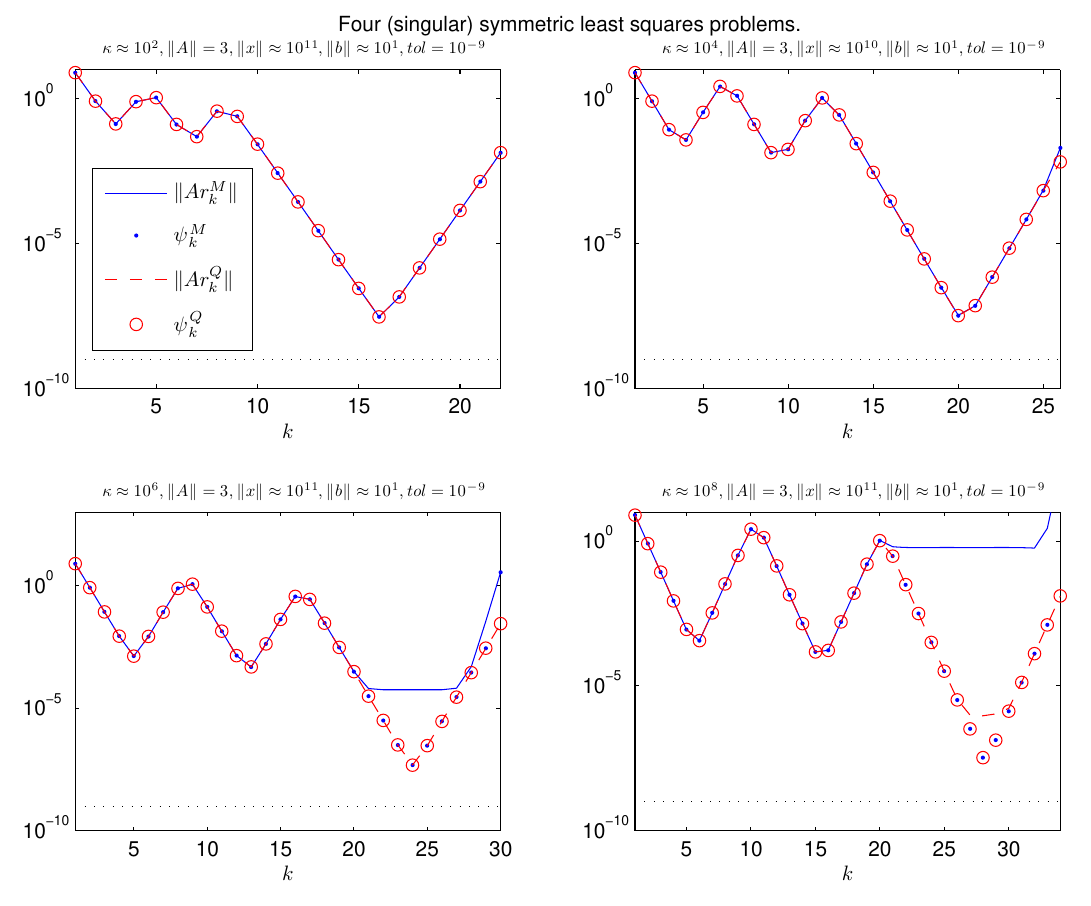}
  \vspace*{-0.25in}
\caption[]{Solving $Ax=b$ with semidefinite $A$
  similar to an example of Sleijpen \etal\ \cite{SVM00}.  $A = Q
  \diag([0_5, \eta, 2\eta, 2 \!:\! \frac{1}{789}\!:\!3])Q$ of
  dimension $n = 797$ with $\norm{A}=3$, where $Q = I - (2/n) ee\T$ is
  a Householder matrix generated by $e = [1, \ldots, 1]\T$.  (We are
  not plotting the NRBE quantities because $\norm{A} \norm{r_k}
  \approx 6$ throughout the iterations in this example.)

%\smallskip

\textbf{Upper left:} $\eta=10^{-2}$ and thus
  $\cond(A) \approx 10^2$.  Also $b=e$ and
  therefore $\norm{x} \gg \norm{b}$.  The graphs of MINRES's
  directly computed $\norm{A r_k^M}$ and recurrently computed
  $\psi_k^M$, and also $\psi_k^Q \approx \norm{Ar_k^Q}$ from
  MINRES-QLP, match very well throughout the iterations.

%\smallskip

\textbf{Upper right:} Here, $\eta=10^{-4}$ and $A$ is more
ill-conditioned than the last example (upper left).  The final MINRES
residual norm $\psi_k^M \approx \norm{A r_k^M}$ is slightly larger
than the final MINRES-QLP residual norm $\psi_k^Q \approx \norm{A
  r_k^Q}$.  The MINRES-QLP $\psi_k^Q$ are excellent approximations of
$\norm{Ar_k^Q}$.

%\smallskip

\textbf{Lower left:}
  $\eta=10^{-6}$ and $\cond(A) \approx 10^6$.
  MINRES's $\psi_k^M$ and $\norm{A r_k^M}$ differ starting at
  iteration 21. Eventually, $\norm{Ar_k^M} \approx 3$, which means
  no convergence.  MINRES-QLP reaches a residual norm of
  $\psi_k^Q = \norm{A r_k^Q} = 10^{-2}$.

%\smallskip

\textbf{Lower right:} $\eta=10^{-8}$.  MINRES performs even worse than
in the last example (lower left).  MINRES-QLP reaches a minimum
$\norm{A r_k^Q} \approx 10^{-7}$ but $\mathit{tol} \!=\! 10^{-8}$
does not shut it down soon enough.
%It ends with a final $\psi_k^Q =
The final $\psi_k^Q =
\norm{A r_k^Q}= 10^{-2}$.  The values of $\psi_k^Q$ and $\norm{A
  r_k^Q}$ differ only at iterations 27--28.

%\smallskip

This figure can be reproduced by \texttt{DPtestLSSing5.m}.
}
\label{figDPtestLSSing1}
\end{figure}

\section{Conclusion}  \label{sec:conclusions}

\MINRES{} constructs its $k$th solution estimate from the recursion
$x_k = D_k t_k = x_{k-1} + \tau_k d_k$ \eqref{minresxk}, where $n$
separate triangular systems $R_k\T D_k\T = V_k^T$ are solved to obtain
the $n$ elements of each direction $d_1, \ldots, d_k$.  (Only $d_k$ is
obtained during iteration $k$, but it has $n$ elements.)

In contrast, \MINRESQLP{} constructs its $k$th solution estimate using
orthogonal steps: $x_k^Q = (V_k P_k)u_k$; see
\eqref{Lsubproblem}--\eqref{qlpeqnsol2}.  Only one triangular system
$L_k u_k = Q_k(\beta_1e_1)$ is involved for each $k$.

Thus \MINRESQLP{} overcomes the potential instability predicted by the
\MINRES{} authors \cite{PS75} and analyzed by Sleijpen \etal\
\cite{SVM00}.  The additional work and storage are moderate, and
maximum efficiency is retained by transferring from \MINRES{} to the
\MINRESQLP{} iterates only when the estimated condition number of $A$
exceeds a specified value.

\MINRES{} and \MINRESQLP{} are readily applicable to Hermitian
matrices, once $\alpha_k$ is typecast as a real scalar in
finite-precision arithmetic.  For both algorithms, we derived
recurrence relations for $\norm{Ar_k}$ and $\norm{Ax_k}$ and used them
to formulate new stopping conditions for singular problems.

TEVD or TSVD are commonly known to use rank-$k$ approximations to $A$
to find approximate solutions to $\min\norm{Ax - b}$ that serve as a
form of \textit{regularization}.  Krylov subspace methods also have
regularization properties \cite{HO93, HN96, KS99}.  Since \MINRESQLP{}
monitors more carefully the rank of $T_k$, which could be $k$ or
$k\!-\!1$, we may say that regularization is a stronger feature in
\MINRESQLP{}, as we have shown in our numerical examples.

It is important to develop robust techniques for estimating
an a priori bound for the solution norm since the \MINRESQLP{} approximations
are not monotonic as is the case in \CG{} and \LSQR{}.  Ideally,
we would also like to determine a practical threshold associated
with the stopping condition $\gamma_k^{(4)} = 0$ in order to
handle cases when $\gamma_k^{(4)}$ is numerically small but not exactly zero.
These are topics for future research.

%\clearpage

\section{Software and reproducible research}  \label{sec:software}

\Matlab~7.6, Fortran~77, and Fortran~90 implementations of \MINRES{} with
new stopping conditions $\norm{Ar_k}\! \leq\! \mathit{tol} \norm{A}
\norm{r_k}$ and $\norm{Ax_k} \leq \mathit{tol} \norm{A} \norm{x_k}$,
and a \Matlab~7.6 implementation of \MINRESQLP{} are available from SOL \cite{SOL}.

Following the philosophy of reproducible computational research as
advocated in \cite{C94, CD02}, for each figure and example in this paper we mention
either the source or the specific \Matlab{} command.  Our \Matlab{}
scripts are available at SOL \cite{SOL}.

\subsection*{Acknowledgements}

We thank Jan Modersitzki, Gerard Sleijpen, Henk Van der Vorst, and also
Kaustuv for providing us with their \Matlab{} scripts, which have
aided us in producing parts of Figures \ref{AnormQLP},
\ref{figDPtestSing3_DP}, and \ref{figDPtestLSSing1}.  We also thank
Michael Friedlander, Rasmus Larsen, and Lek-Heng Lim for their finest
examples of work, discussion, and support.
%SC: Thank reviewers
We are most grateful to our anonymous reviewers for their insightful
suggestions for improving this manuscript.
Last but not least, we dedicate this paper to the memory of our
colleague and friend, Gene Golub.

\clearpage

\appendix

\section{Proof that
   {\rm $T_\ell$ is nonsingular iff $b \in \range(A)$
   (section \ref{sec:Lanproperties})}}
\label{appendixA}

If $T_\ell$ is nonsingular, we have $A V_\ell
  T_\ell\inv e_1 = V_\ell e_1 = \beta_1^{-1} b$.  Conversely,
  if $b \in \range(A)$, then $\range(V_\ell) \subseteq \range(A)$
  and $ \Null(A) \cap \range(V_\ell) = \{0\}$. We also know that
  $\rank(V_\ell) =\ell$ and $\rank(T_\ell) = \rank (V_\ell T_\ell) =
  \rank (A V_\ell) = \rank(V_\ell) - \dim[ \Null(A) \cap \range(V_\ell)]$;
  see \cite[Fact 2.10.4 ii]{B2}.
  Thus $\rank(T_\ell) = \ell$ and so $T_\ell$ is nonsingular.)

\frenchspacing
\bibliographystyle{abbrv}
\bibliography{refs5}

\end{document}